\documentclass[1p,times,longtitle]{elsarticle}

\usepackage{amsmath} % for \begin{bmatrix}, \begin{smallmatrix}, \begin{align} etc.
\usepackage{amsfonts} % for \mathbb{} etc. `blackboard' for uppercase only.
\usepackage{amsthm}
\usepackage{ifpdf}
\ifpdf
\usepackage[%
  pdftitle={Angles Between Infinite Dimensional Subspaces
with Applications to the Rayleigh-Ritz and Alternating Projectors Methods},%
  pdfauthor={Andrew Knyazev, Abram Jujunashvili, and Merico Argentati},%
  pdfsubject={Angles Between Infinite Dimensional Subspaces
with Applications to the Rayleigh-Ritz and Alternating Projectors Methods},%
  pdfkeywords={Hilbert space, angles between subspaces, gap, minimum gap, principal subspaces,
invariant subspace, orthogonal projector, isometry, polar decomposition,
spectral theory, canonical correlations,
Rayleigh-Ritz method, Ritz values, perturbation, Hausdorff distance, alternating projectors,
conjugate gradient, domain decomposition},%
  pdfstartview=FitH,%
  bookmarks=true,%
  bookmarksopen=true,%
  breaklinks=true,%
  colorlinks=true,%
  linkcolor=blue,anchorcolor=blue,%
  citecolor=blue,filecolor=blue,%
  menucolor=blue,pagecolor=blue,%
  urlcolor=blue]{hyperref}
\else
\usepackage[%
  breaklinks=true,%
  colorlinks=true,%
  linkcolor=blue,anchorcolor=blue,%
  citecolor=blue,filecolor=blue,%
  menucolor=blue,pagecolor=blue,%
  urlcolor=blue]{hyperref}
\fi

\newtheorem{theorem}{Theorem}[section]
\newtheorem{lemma}[theorem]{Lemma}

\newtheorem{definition}[theorem]{Definition}

\newtheorem{remark}[theorem]{Remark}

\newtheorem{cor}[theorem]{Corollary}

\numberwithin{equation}{section}

\newcommand{\sis}{{\sigma}^2}
\newcommand{\si}{\sigma}
\newcommand{\Si}{\Sigma}
\newcommand{\ff}{{\mathcal {F}}}
\def\gg{{\mathcal {G}}}         %?
\newcommand{\fg}{{P_{\mathcal{F}}P_{\mathcal{G}}}}
\newcommand{\gf}{{P_{\mathcal{G}}P_{\mathcal{F}}}}
\newcommand{\fgf}{{P_{\mathcal{F}}P_{\mathcal{G}}P_{\mathcal{F}}}}
\newcommand{\gfg}{{P_{\mathcal{G}}P_{\mathcal{F}}P_{\mathcal{G}}}}
\newcommand{\pf}{P_{\mathcal{F}}}
\newcommand{\pg}{P_{\mathcal{G}}}
\newcommand{\fo}{{\mathcal{F}}^ \bot}
\newcommand{\go}{{\mathcal{G}}^ \bot}

\newcommand{\pu}{P_{\mathcal{U}}}
\newcommand{\pv}{P_{\mathcal{V}}}
\newcommand{\hh}{{\mathcal {H}}}
\newcommand{\uu}{{\mathcal {U}}}
\newcommand{\vv}{{\mathcal {V}}}

\newcommand{\bb}{{\mathcal {B}}}

\newcommand{\pgt}{P_{\tilde{\mathcal{G}}}}
\newcommand{\fft}{\tilde{{\mathcal {F}}}}
\newcommand{\ggt}{\tilde{{\mathcal {G}}}}

\newcommand{\pgo}{P_{{\mathcal{G}}^ \bot}}
\newcommand{\pfo}{P_{{\mathcal{F}}^ \bot}}
\newcommand{\moo}{\mathfrak{M}_{00}}
\newcommand{\moi}{\mathfrak{M}_{01}}
\newcommand{\mio}{\mathfrak{M}_{10}}
\newcommand{\mii}{\mathfrak{M}_{11}}
\newcommand{\mf}{\mathfrak{M}_{\ff}}%{\mathfrak{M}_{0}}
\newcommand{\mfo}{\mathfrak{M}_{\fo}}%{\mathfrak{M}_{1}}
\newcommand{\mg}{\mathfrak{M}_{\gg}}
\newcommand{\mgo}{\mathfrak{M}_{\go}}
\newcommand{\mm}{\mathfrak{M}}
\newcommand{\mum}{{\mu}^2}
\newcommand{\ulo}{u_{\bot}}
\newcommand{\vlo}{v_{\bot}}
\newcommand{\luo}{{\mathcal{U}}_ \bot}
\newcommand{\lvo}{{\mathcal{V}}_ \bot}

\newcommand{\rr}{\mathfrak{R}}
\newcommand{\nn}{\mathfrak{N}}
\newcommand{\dd}{\mathfrak{D}}

\newcommand{\Span}{\rm span}
\newcommand{\gap}{\rm gap}
\newcommand{\dist}{\rm dist}

\journal{Journal of Functional Analysis}

\begin{document}
\begin{frontmatter}
\title%[Angles Between Infinite Dimensional Subspaces with Applications]
{Angles Between Infinite Dimensional Subspaces
with Applications to the Rayleigh-Ritz 
and Alternating Projectors Methods\tnoteref{label1}}
\tnotetext[label1]{A preliminary version is available at 
\url{http://arxiv.org/abs/0705.1023}}
\author[UCD]{{Andrew Knyazev}
\corref{cor1}\fnref{label2}}
\ead{andrew[dot]knyazev[at]ucdenver[dot]edu}
\ead[url]{http://math.ucdenver.edu/\~{}aknyazev/}
\fntext[label2]{This material is based upon work supported by the NSF DMS award 0612751.}
\cortext[cor1]{Corresponding author. Email: andrew[dot]knyazev[at]ucdenver[dot]edu}
\address[UCD]{Department of Mathematical and Statistical Sciences
University of Colorado Denver,
P.O. Box 173364, Campus Box 170, Denver, CO 80217-3364}
%\fnref{label3}}
%\fntext[label3]{}
\author[UCD]{Abram Jujunashvili}
\ead{Abram[dot]Jujunashvili[at]na-net[dot]ornl[dot]gov}
\author[UCD]{Merico Argentati}
\ead{Merico[dot]Argentati[at]na-net[dot]ornl[dot]gov}
\ead[url]{http://math.ucdenver.edu/\~{}margenta/}
\begin{abstract}
We define angles from-to and between infinite dimensional subspaces of a Hilbert space, 
inspired by the work of E.~J.~Hannan,~1961/1962 for general
canonical correlations of stochastic processes.
The spectral theory of selfadjoint operators is used to investigate
the properties of the angles, e.g., to establish connections between the angles corresponding to
orthogonal complements. The classical gaps and angles of Dixmier and Friedrichs are characterized 
in terms of the angles. We introduce principal invariant subspaces and prove that they
are connected by an isometry that appears in the polar decomposition of the product of corresponding
orthogonal projectors. Point angles are defined by analogy with the point operator spectrum.
We bound the Hausdorff distance between the sets of 
the squared cosines of the angles corresponding
to the original subspaces and their perturbations.
We show that the squared cosines  of the angles from one subspace to another
can be interpreted as Ritz values in the Rayleigh-Ritz method, where the former subspace serves as a
trial subspace and the orthogonal projector of the latter subspace serves as an operator in the
Rayleigh-Ritz method. The Hausdorff distance between the Ritz values, corresponding
to different trial subspaces, is shown to be bounded by a constant times the gap between the trial subspaces.
We prove a similar eigenvalue perturbation bound that involves the gap squared.
Finally, we consider the classical alternating projectors method
and propose its ultimate acceleration, using the conjugate gradient approach.
The corresponding convergence rate estimate is obtained in terms of the angles.
We illustrate a possible acceleration for the domain decomposition method with a small
overlap for the 1D diffusion equation.\\
\noindent\copyright{2010}{ %Elsevier B.V. 
Knyazev, Jujunashvili, and Argentati.
All rights reserved.}
\end{abstract}
\end{frontmatter}
\newpage
\section{Introduction}
Principal angles, also referred to as canonical angles, or simply as angles, 
between subspaces represent one of the classical mathematical tools
with many applications. The cosines of the angles are related to 
canonical correlations which are widely used in statistics.
Angles between finite dimensional subspaces have become so 
popular that they can be found even in linear algebra textbooks. 

The angles between subspaces $\ff$ and $\gg$ are defined as $q=\min\{\dim\ff,\dim\gg\}$
values on $[0,\pi/2]$ if $q<\infty$. In the case $q=\infty$, where 
both subspaces $\ff$ and $\gg$ are infinite dimensional, traditionally only
single-valued angles are defined, which in the case  $q<\infty$ would correspond to 
the smallest (Dixmier \cite{Dixmier}), smallest non-zero  
(Friedrichs \cite{Friedrichs}), or largest (Krein et al. \cite{kkm48}), angles. 
We define angles from-to
and between (infinite) dimensional subspaces of a Hilbert space
using the spectra of the product of corresponding orthogonal
projectors. The definition is consistent with the finite dimensional case  $q<\infty$
and results in a \emph{set}, possibly infinite, of angles. 

Our definition is inspired by E.J.~Hannan \cite{hannan},
where such an approach to canonical correlations
of stochastic processes is suggested.
Canonical correlations for stochastic processes and functional
data often involve infinite dimensional subspaces.
This paper is intended to revive the interest in angles between
infinite dimensional subspaces.

In functional analysis, the gap and the minimum gap are important concepts used, 
e.g.,\ in operator perturbation theory (\cite{kato}).
The gap between infinite dimensional subspaces
bounds the perturbation of a closed linear operator by measuring
the change in its graph. We show in Theorem \ref{gap_angles} that 
the gap is closely connected to the sine of the largest angle. 

The minimum gap between infinite
dimensional subspaces provides a necessary and sufficient condition
to determine if the sum of two subspaces is closed. The minimum gap is applied, e.g.,\ 
in \cite{cmame07} to prove wellposedness of degenerate saddle point problems. 
The minimum gap is precisely, see Theorem \ref{angles_kato},  
the sine of the angle of Friedrichs, which, in its turn, as  
shown in Theorem \ref{angles_deutsch}, is the infimum of the set of nonzero angles. 
The Dixmier angle is simply the smallest of all angles in our definition. 

We consider a
(real or complex) Hilbert space equipped with an inner product
$(f,g)$  and a vector norm $\|f\|=(f,f)^{1/2}$.
The angle between two unit vectors $f$ and $g$ is
$\theta(f,g)=\arccos|(f,g)|\in[0,\pi/2].$
In \S \ref{angles_properties} of the present paper, we replace
$1D$ subspaces spanned by the vectors $f$ and $g$  with (infinite dimensional) subspaces,
and introduce the concept of principal angles from one subspace to another and
between subspaces using the spectral theory of
selfadjoint operators. We investigate the basic properties of the angles,
which are already known for finite dimensional subspaces, see \cite{ka02},
e.g., we establish connections between the angles corresponding to
subspaces and their orthogonal complements.
We express  classical quantities: the
gap and the minimum gap between subspaces, in terms of the angles.

In \S~\ref{angles_properties}, 
we provide a foundation and give necessary tools for the rest of the paper, 
see also \cite{bs2010} and references there.  
In \S~\ref{princ_vectors_1}, we
introduce principal invariant subspaces and prove that they are
connected by the isometry that appears in the polar decomposition
of the product of corresponding orthogonal projectors. We define
point angles %and their multiplicity
by analogy with the point operator spectrum and
consider peculiar properties of the invariant subspaces
corresponding to a point angle. In \S~\ref{prox_angl}, 
the Hausdorff distance is used to measure the change in the principal angles,
where one of the subspaces varies, extending some of our previous results
of \cite{ka02,ka05} to infinite dimensional subspaces.

We consider two applications of the angles: to bound the change in
Ritz values, where the Rayleigh-Ritz method is applied to different
infinite dimensional trial subspaces, in \S \ref{ritz_val};
and to analyze and accelerate the convergence of the
classical alternating projectors method (e.g.,\ \cite[Chapter IX]{MR1823556})
in the context of a specific example---a domain decomposition method (DDM) with an overlap, 
in \S \ref{DDM}.
In computer simulations the subspaces involved are evidently
finite dimensional; however, the assumption of the finite dimensionality
is sometimes irrelevant in theoretical analysis of the methods.

In \S \ref{ritz_val}, we consider
the Rayleigh-Ritz method for a bounded selfadjoint operator $A$
on a trial subspace $\ff$ of a Hilbert space, where
the spectrum  $\Si((P_{\ff}A)|_{\ff})$
of the restriction to the subspace $\ff$ of the
product of the orthoprojector $P_\ff$ onto $\ff$ and the operator $A$
is called the set of Ritz values, corresponding to $A$ and $\ff$.
In the main result of \S \ref{ritz_val}, we bound the change in the Ritz values,
where one trial subspace $\ff$ is replaced with another subspace $\gg$,
using the Hausdorff distance between the sets of Ritz values,
by the spread of the spectrum times the gap between the subspaces.
The proof of the general case is based on a specific case
of one dimensional subspaces $\ff$ and $\gg$,
spanned by unit vectors $f$ and $g,$ correspondingly,
where %the sets of Ritz values turn into the Rayleigh quotients and
the estimate becomes particularly simple:
%\begin{equation}\label{angle_RR}
$
\left|(f,Af)-(g,Ag)\right|\leq
\left(\lambda_{\max} - \lambda_{\min}\right)\sin(\theta(f,g));
$
%\end{equation}
here $\lambda_{\max} - \lambda_{\min}$ is the spread of the spectrum of $A$, cf. \cite{ka03}.
If in addition $f$ or $g$ is an eigenvector of $A$, the same bound holds but with the sine squared---similarly,
our Hausdorff distance bound involves the gap squared, assuming that one of the trial
subspaces is $A$-invariant.
The material of  \S \ref{ritz_val} generalizes some of the earlier results of
\cite{ka05,ka07} and \cite{ko06} for the finite dimensional case.
The Rayleigh-Ritz method
with infinite dimensional trial subspaces is used in
the method of intermediate problems for determining
two-sided bounds for eigenvalues, e.g.,\
\cite{MR0400004,MR0477971}. The results of \S \ref{ritz_val}
may be useful in obtaining {\it a priori} estimates of
the accuracy of the method of intermediate problems, but
this is outside of the scope of the present paper.

Our other application, in \S~\ref{DDM}, is the classical alternating projectors method:
$e^{(i+1)} = \fg e^{(i)},$ $e^{(0)} \in \ff$, where
$\ff$ and $\gg$ are two given subspaces
and  $\pf$ and $\pg$ are the orthogonal projectors onto $\ff$
and $\gg$, respectively.
If $\left\|\left.(\fg)\right|_\ff\right\|<1$ then the sequence of vectors
 $ e^{(i)}$ evidently converges to zero. Such a situation is typical if
$ e^{(i)}$ represents an error of an iterative method,
e.g., a multiplicative DDM,
%originated from \cite{Schwarz:1890:GMA},
so that the alternating projectors method
describes the error propagation in the DDM, e.g.,\  \cite{MR992008,MR1009556}.

If the intersection $\ff \cap \gg$ is nontrivial
then the sequence of vectors $ e^{(i)}$ converges under reasonable assumptions to
the orthogonal projection of $ e^{(0)}$ onto $\ff \cap \gg$ as in
the von Neumann-Halperin method, see \cite{MR0032011,MR0141978}, and \cite{MR1990157}.
Several attempts to estimate and accelerate the convergence of
alternating projectors method are made, e.g.,\
\cite{MR1853223,MR1990157}, and \cite{MR1896233}.
Here, we use a different approach,
known in the DDM context, e.g.,\  \cite{MR992008,MR1009556},
but apparently novel in the context of the von Neumann-Halperin method,
and suggest the ultimate, conjugate gradient based,
acceleration of the von Neumann-Halperin alternating projectors method.

Our idea of the acceleration is inspired by the following facts.
On the one hand, every selfadjoint non-negative non-expansion $A, \, 0\leq A \leq I$ in
a Hilbert space $\hh$ can be extended to an orthogonal projector $\pg$ in the space
$\hh\times \hh$, e.g., \cite{halmos69,RiSN}, and, thus, is unitarily equivalent to
a product of two orthogonal projectors $\fg$ restricted to the subspace
$\ff = \hh \times \{0\}$.
Any polynomial iterative method that involves as a main step a multiplication of
a vector by $A$ can thus be called an ``alternating projectors'' method.
On the other hand, the conjugate gradient method
is the optimal polynomial method for computing the null space of $A,$
therefore the conjugate gradient approach provides the ultimate acceleration
of the alternating projectors method.

We give in \S \ref{DDM} the corresponding convergence rate estimate in terms of the angles.
We illustrate a possible acceleration for the DDM with a small
overlap for the 1D diffusion equation. The convergence of the classical
alternating projectors method degrades when the overlap gets smaller, but the
conjugate gradient method we describe converges to the exact solution
in two iterations. For a finite difference approximation
of the 1D diffusion equation a similar result can be found in \cite{MR2058877}.

This paper is partially based on \cite{juju_thesis}, where
simple proofs that we skip here can be found.
\section{Definition and Properties of the Angles}\label{angles_properties}
Here we define angles from one subspace to another and angles between subspaces,
and investigate the properties of the (sets of) angles,
such as the relationship concerning angles
between the subspaces
and their orthogonal complements. We express the gap and the minimum
gap between subspaces in terms of angles. We introduce principal invariant subspaces and
prove that they are connected by an isometry that appears in the polar decomposition of the
product of corresponding orthogonal projectors. We define point angles
and their multiplicities by analogy with the
point operator spectrum, and consider peculiar properties of the invariant subspaces corresponding
to a point angle.
\subsection{Preliminaries}\label{prelims}
Let $\hh$ be a (real or complex) Hilbert space and let $\ff$ and $\gg$ be proper nontrivial subspaces.
A subspace is defined as a {\em closed}
linear manifold. Let $\pf$ and $\pg$ be the orthogonal projectors onto $\ff$
and $\gg$, respectively. We denote by $\bb(\hh)$ the Banach space of bounded linear operators defined
on $\hh$ with the induced norm.
We use the same notation $\|\cdot\|$
for the vector norm on $\hh$, associated with the inner product $(\cdot,\cdot)$ on $\hh$,
as well as for the induced operator norm on $\bb(\hh)$.
For $T\in\bb(\hh)$ we define $|T|=\sqrt{T^*T}$, using the positive square root.
$T|_U$ denotes the restriction of the operator T to its invariant subspace $U$.
By $\dd(T)$, $\rr(T)$, $\nn(T)$, $\Si(T)$, and  $\Si_p(T)$ we denote the domain, range,
null space, spectrum, and point spectrum, respectively, of the operator $T$.
In this paper, we distinguish
only between finite and infinite dimensions. If $q$ is a finite number
then we set by definition $\min \{q, \, \infty \}=q$ and $\max
\{q, \, \infty \}= \infty,$ and assume that $\infty \leq \infty$ holds.
We use $\oplus$ to highlight that the sum of subspaces is orthogonal and for
the corresponding sum of operators.
We denote the $\ominus$ operation between subspaces $\ff$ and $\gg$ by 
$\ff\ominus\gg=\ff\cap\gg^\perp.$

Introducing an orthogonal decomposition
$\hh=\moo\oplus\moi \oplus \mio \oplus \mii \oplus \mm,$
where
\[
\moo=\ff \cap \gg, \enspace \moi=\ff \cap \go, \enspace \mio=\fo \cap \gg, \enspace
\mii=\fo \cap \go,
\]
(see, e.g.,\ \cite{{halmos69},{davis58}}),
we note that every subspace in the decomposition is
$\pf$ and $\pg$ invariant.

\begin{definition}(See \cite{halmos69}).  
Two subspaces $\ff\subset\hh$ and $\gg\subset\hh$ are said to be 
in \emph{generic position} within the space $\hh$, if all four subspaces 
$\moo,\,\moi,\,\mio$, and $\mii$ are null-dimensional. 
\end{definition}
Clearly, subspaces $\ff\subset\hh$ and $\gg\subset\hh$ are 
in generic position within the space $\hh$ iff 
any of the pairs of 
subspaces: $\ff\subset\hh$ and $\go\subset\hh$, or 
$\fo\subset\hh$ and $\gg\subset\hh$, or 
$\fo\subset\hh$ and $\go\subset\hh$,
is in generic position within the space $\hh$.

The fifth part, $\mm$, can be further orthogonally split in two different ways as follows:
\begin{itemize}
\item
$\mm=\mf\oplus\mfo$ with 
$\mf=\ff\ominus(\moo\oplus\moi),\enspace \mfo=\fo\ominus(\mio\oplus\mii),$ or 
\item
$\mm=\mg\oplus \mgo$ with
$\mg=\gg\ominus (\moo\oplus\mio),\enspace\mgo=\go\ominus(\moi\oplus\mii).$
\end{itemize}
We obtain orthoprojectors' decompositions %of the orthoprojectors
\[
\pf=I_{\moo}\oplus I_{\moi}\oplus 0_{\mio}\oplus  0_{\mii} \oplus \pf|_{\mm}
\text{ and }
\pg=I_{\moo}\oplus 0_{\moi}\oplus I_{\mio}\oplus  0_{\mii} \oplus \pg|_{\mm},
\]
and decompositions of their products:
\[
(\fg)|_\ff=I_{\moo}\oplus 0_{\moi} \oplus (\fg)|_{\mf},  
\text{ and }
(\gf)|_\gg=I_{\moo}\oplus 0_{\mio} \oplus (\gf)|_{\mg}.
\]
These decompositions are very useful in the sequel. In the next theorem
we apply them to prove the unitary equivalence of the
operators $\fgf$ and $\gfg$.
\begin{theorem}\label{andrew_ilya}
Let $\ff$ and $\gg$ be subspaces of $\hh$. Then there exists a unitary operator $W\in\bb(\hh)$
such that $\fgf = W^* \gfg W.$
\end{theorem}
\begin{proof}
%(With contribution of I. Lashuk, UC Denver).
Denote $T=\gf$. Then $T^*=\fg$ and 
$T^*T=\fgf.$ Using, e.g.,\
\cite[\S 110, p. 286]{RiSN} or \cite[\S VI.2.7, p. 334]{kato}, 
we introduce the polar decomposition,  
 $T=U|T|$, where $|T|=\sqrt{T^*T}=\sqrt{\fgf}$ is selfadjoint and nonnegative 
and $U:\,{\rr(|T|)}\to{\rr(T)}$ is an isometry. 
We extend $U$ by continuity, keeping the same notation, to 
the isometry $U:\,\overline{\rr(|T|)}\to\overline{\rr(T)}$. 
It is easy to check directly that $\nn(|T|)=\nn(T)$, so 
$\overline{\rr(|T|)}=(\nn(T))^\perp$ since $|T|$ is selfadjoint. 
Taking also into account that  $\overline{\rr(T)}=(\nn(T^*))^\perp$, 
we have $U:\,(\nn(T))^\perp\to(\nn(T^*))^\perp$. 

For a general operator 
$T\in\bb(\hh)$, the isometry $U$ is then typically extended to a partial isometry 
$U\in\bb(\hh)$ by setting $U=0$ on $\nn(T)$. For our special $T=\gf$, 
we can do better and extend $U$ to a unitary operator $W\in\bb(\hh)$. 
Indeed, we set $W=U$ on $(\nn(T))^\perp$ to make $W$ an extension of $U$. 
To make $W$ unitary, we set $W=V$ on $\nn(T)$, where 
$V:\,\nn(T)\to\nn(T^*)$ must be an isometry. 
The specific form of $V$ is of no importance, since it evidently does not affect 
the validity of the formula $\gf=W\sqrt{\fgf}$, which implies 
$\fg=\sqrt{\fgf}W^*$. Multiplying these equalities we obtain the required 
$\gfg=W\fgf W^*$.

For the existence of such $V$, it is  
sufficient (and, in fact,  necessary) 
that $\nn(T^*)=\nn(\fg)$ and $\nn(T)=\nn(\gf)$ be isomorphic.
Using the five-parts decomposition,
we get 
\[\nn(\fg)=\moi \oplus \mio \oplus \mii \oplus \nn((\fg)|_\mm),\,
\nn(\gf)=\moi \oplus \mio \oplus \mii\oplus \nn((\gf)|_\mm).
\]
The first three terms in the decompositions of $\nn(\fg)$ and $\nn(\gf)$
are the same, so $\nn(\fg)$ and $\nn(\gf)$ are isomorphic iff
the last terms $\nn((\fg)|_\mm)=\mgo$ and $\nn((\gf)|_\mm)=\mfo$ are
isomorphic. The subspaces $\mf=\pf\mm\subseteq\mm$ and $\mg=\pg\mm\subseteq\mm$
are in generic position within the space $\mm$, see
\cite{halmos69}, as well as their orthogonal in $\mm$ complements
$\mfo$ and $\mgo$. According to \cite[Proof of Theorem 1, p. 382]{halmos69},
any two subspaces in generic position are isomorphic, thus 
$\nn(\fg)$ and $\nn(\gf)$ are isomorphic.
\end{proof}
\begin{cor}\label{equiv_restric}
The operators $(\fg)|_{\mf}$ and $(\gf)|_{\mg}$ are unitarily equivalent.
\end{cor}
\begin{proof}
%The decomposition of $\hh$ gives
We have that
$\fgf=(\fg)|_{\mf} \oplus I_{\moo} \oplus 0_{\hh \ominus (\moo \oplus \mf)}$ and
$\gfg=(\gf)|_{\mg} \oplus I_{\moo} \oplus 0_{\hh \ominus (\moo \oplus \mg)}$.
The subspaces $\mf$ and $\mg$ are connected by $\mf=W \mg$, $\mg=W^* \mf$,
and $\fgf = W^* \gfg W$. %where $W$ comes from Theorem \ref{andrew_ilya}.
\end{proof}

%\begin{remark}%\label{for_RS_K_D_1}
In the important particular case $\|\pf-\pg\|<1$, subspaces $\ff$ and $\gg$ are isometric and 
\citet[\S VII.105]{RiSN} explicitly describe a partial isometry 
\[U=\pg[I+\pf(\pg-\pf)\pf]^{-1/2}\pf\] 
that maps $\ff$ one-to-one and onto $\gg$. On $\ff$, clearly 
$I+\pf(\pg-\pf)\pf$ is just the same as $\fgf,$ so this $U$ represents the 
partial isometry in the polar decomposition in the proof of our Theorem \ref{andrew_ilya},  
in this case. Let 
\[V=(I-\pg)[I+(I-\pf)((I-\pg)-(I-\pf))(I-\pf)]^{-1/2}(I-\pf)\]
be another partial isometry that maps $\fo$ one-to-one and onto $\go$, 
constructed in the same way as $U$.
Setting $W=U+V$, we extend $U$  
from the subspace $\ff$ to a unitary operator $W$ on the whole space. 
The sum $W=U+V$ is the same as the unitary extension suggested in 
\citet[\S I.4.6, \S I.6.8]{kato} and \citet{DavKah}:
\begin{eqnarray}\label{dav_ka_ka}
W&=&[\gf+(I-\pg)(I-\pf)][I-(\pf-\pg)^2]^{-1/2}\\
&=&[(I-\pg)(I-\pf)+\pf\pg]^{-1/2}[\pg\pf+(I-\pg)(I-\pf)]\nonumber
\end{eqnarray}%
(the second equality holds since the corresponding terms in square brackets are 
the same and $(\pf-\pg)^2$ commutes both with $\pf$ and $\pg$), which is used there 
to prove the unitary equivalence $\pf=W^*\pg W$. It is easy to check 
directly that the operator $W$ is unitary and that on $\ff$ it acts the same as 
the operator $U$, so it is indeed 
a unitary extension of $U$. If $\|\pf-\pg\|<1$,  
Theorem \ref{andrew_ilya} holds with this choice of $W$. 
%\end{remark}

In the next subsection we define angles from-to and between subspaces using 
the spectrum of the product of two orthogonal projectors. 
Our goal is to develop a theory of angles 
from-to and between subspaces based on the well-known spectral theory 
of selfadjoint bounded operators. 
\subsection{Angles From--To and Angles Between Subspaces}\label{from_between}
%We define the angles from one subspace to another and the
%angles between subspaces using the spectra of the products of
%projectors, and we investigate the basic properties of angles.
\begin{definition}\label{angles_from_cos}
$
\hat{\Theta}(\ff,\gg)=\{\theta:\theta=\arccos(\si), \si\geq0,
\sis\in\Sigma((\fg)|_{\ff}) \} \subseteq [0,{\pi}/{2}]
$
is called the set of angles \emph{from} the subspace $\ff$ to the subspace $\gg$. 
Angles 
$\Theta(\ff, \gg)=\hat{\Theta}(\ff, \gg) \cap \hat{\Theta}(\gg, \ff)$
are called angles \emph{between} the subspaces $\ff$ and $\gg$.
\end{definition}
Let the operator $T \in \bb(\hh)$
be a selfadjoint nonnegative contraction. Using an extension of
$T$ to an orthogonal projector \cite[\S A.2, p. 461]{RiSN}, there
exist subspaces $\ff$ and $\gg$ in $\hh^2$ such that $T$ is
unitarily equivalent to $(\fg)|_{\ff}$, where $\pf$ and $\pg$ are
the corresponding orthogonal projectors in $\hh^2$. This implies
that the spectrum of the product of two orthogonal projectors is as
general a set as the spectrum of an arbitrary selfadjoint
nonnegative contraction,
%i.e.,\ it can be any closed subset of $[0,1]$, e.g.,\ \cite{?}.
so the set of angles between subspaces can be a sufficiently general
subset of $[0,{\pi}/{2}]$. 
\begin{definition}\label{multiplicity_of_angle}
The angles 
$
\hat{\Theta}_p(\ff, \gg)= \left\{\theta\in\hat{\Theta}(\ff,\gg):
{\cos}^2 (\theta)\in\Sigma_p\left((\fg)|_\ff\right)\right\}$
and
$\Theta_p(\ff,\gg)=\hat{\Theta}_p(\ff,\gg)\cap\hat{\Theta}_p(\gg,\ff)$
are called {\em point} angles.
Angle $\theta\in\hat\Theta_p(\ff,\gg)$ inherits its multiplicity from
${\cos}^2 (\theta)\in\Sigma_p\left((\fg)|_\ff\right)$.
Multiplicity of angle $\theta\in\Theta_p(\ff,\gg)$
is the minimum of multiplicities of $\theta\in\hat\Theta_p(\ff,\gg)$
and $\theta\in\hat\Theta_p(\gg,\ff)$.
\end{definition}

For two vectors $f$ and $g$ in the plane, and their orthogonal
counterparts $f^\perp$ and $g^\perp$ we evidently have
that  $\theta(f,g) = \theta(f^\perp,g^\perp)$ and
 $\theta(f,g) + \theta(f,g^\perp)=\pi/2.$
We now describe relationships for angles, corresponding to
 subspaces $\ff,\gg,\fo,$ and $\go.$
We first consider the angles from one subspace to another
as they reveal the finer details and provide a foundation for
statements on angles between subspaces.
\begin{theorem}\label{7_relations}
For any pair of subspaces $\ff$ and $\gg$ of $\hh$:
\begin{enumerate}
\item $\hat{\Theta}(\ff, \go)={\pi}/{2}-\hat{\Theta}(\ff, \gg)$;
\item $\hat{\Theta}(\gg, \ff) \setminus \{{\pi}/{2} \} =
       \hat{\Theta}(\ff, \gg) \setminus \{{\pi}/{2} \}$;
\item $\hat{\Theta}(\fo, \gg) \setminus ( \{0\} \cup \{{\pi}/{2} \} ) =
       {\pi}/{2}- \{\hat{\Theta}(\ff, \gg) \setminus ( \{0\} \cup \{{\pi}/{2} \} ) \}$;
\item $\hat{\Theta}(\fo, \go) \setminus ( \{0\} \cup \{{\pi}/{2} \} ) =
       \hat{\Theta}(\ff, \gg) \setminus ( \{0\} \cup \{{\pi}/{2} \} )$;
\item $\hat{\Theta}(\gg, \fo) \setminus \{ 0 \} =
       {\pi}/{2}- \{\hat{\Theta}(\ff, \gg) \setminus \{ {\pi}/{2} \} \}$;
\item $\hat{\Theta}(\go, \ff) \setminus \{{\pi}/{2} \} =
       {\pi}/{2}- \{\hat{\Theta}(\ff, \gg) \setminus \{ 0 \} \}$;
\item $\hat{\Theta}(\go, \fo) \setminus \{ 0 \} =
       \hat{\Theta}(\ff, \gg) \setminus \{ 0 \}$.
\end{enumerate}
\begin{table}[ht]
\caption{Multiplicities of $0$ and $\pi/2$ angles for different pairs of subspaces}\label{t1}
\renewcommand\arraystretch{1.5}
\noindent\[
\begin{array}{|c|c|c||c|c|c|}
\hline
{\text{Pair}}&{\theta=0}&{\theta=\pi/2}&{\text{Pair}}&{\theta=0}&{\theta=\pi/2}\\
\hline
{\hat{\Theta}(\ff, \gg)}&{\dim \moo}&{\dim \moi}&{\hat{\Theta}(\gg, \ff)}&{\dim \moo}&{\dim \mio}\\
\hline
{\hat{\Theta}(\ff, \go)}&{\dim \moi}&{\dim \moo}&{\hat{\Theta}(\gg, \fo)}&{\dim \mio}&{\dim \moo}\\
\hline
{\hat{\Theta}(\fo, \gg)}&{\dim \mio}&{\dim \mii}&{\hat{\Theta}(\go, \ff)}&{\dim \moi}&{\dim \mii}\\
\hline
{\hat{\Theta}(\fo, \go)}&{\dim \mii}&{\dim \mio}&{\hat{\Theta}(\go, \fo)}&{\dim \mii}&{\dim \moi}\\
\hline
\end{array}
\]
\end{table}
The multiplicities of the point angles $\theta \in (0, {\pi}/{2})$ in $\hat{\Theta}(\ff, \gg)$,
$\hat{\Theta}(\fo, \go)$, $\hat{\Theta}(\gg, \ff)$ and $\hat{\Theta}(\go, \fo)$ are the same, and are equal
to the multiplicities of the point angles $ {\pi}/{2} - \theta  \in (0, {\pi}/{2})$ in
$\hat{\Theta}(\ff, \go)$, $\hat{\Theta}(\fo, \gg)$, $\hat{\Theta}(\gg, \fo)$ and
$\hat{\Theta}(\go, \ff).$
\end{theorem}
\begin{proof}
(1) Using the equalities $(\pf \pgo)|_\ff = \pf|_\ff - (\fg)|_\ff
=I|_\ff - (\fg)|_\ff$ and the spectral mapping theorem
%, e.g.,\ \cite[Corollary 1, p. 227]{yosida},  
for $f(T)=I-T$ we have
$\Si((\pf \pgo)|_\ff)=1-\Si((\fg)|_\ff)$. Next, using the identity
$\nn(T-\lambda I) = \nn((I-T)-(1-\lambda)I)$, we conclude that
$\lambda$ is an eigenvalue of $(\fg)|_{\ff}$ if and only if
$1-\lambda$ is an eigenvalue of $(\pf \pgo)|_\ff$, and that their
multiplicities are the same.
\\ %Using Definitions \ref{angles_from_cos} and \ref{multiplicity_of_angle} completes the proof.
(2) The statement on nonzero angles follows
from Corollary \ref{equiv_restric}. The part concerning the zero angles follows from the
fact that $(\fg)|_{\moo}= (\gf)|_{\moo}=I|_{\moo}$.
\\ %Differences in the multiplicities of the right angles are given in Table~\ref{t1}.\\
(3--7) All other statements can be obtained from
the (1--2) by exchanging the subspaces.
Table~\ref{t1} entries are checked directly using the five-parts decomposition.
\end{proof}

Theorem \ref{angles_different_pairs} and Table~\ref{t2}
relate the sets of angles between pairs of subspaces:
\begin{theorem}\label{angles_different_pairs}
For any subspaces $\ff$ and $\gg$ of $\hh$ the following equalities hold:
\begin{enumerate}
\item $\Theta(\ff, \gg) \setminus ( \{0\} \cup \{{\pi}/{2} \} ) =
\{ {\pi}/{2} - \Theta(\ff, \go) \} \setminus ( \{0\} \cup \{{\pi}/{2} \} );$
\item $\Theta(\ff, \gg) \setminus \{0 \} =
\Theta(\fo, \go) \setminus \{0 \};$
\item $\Theta(\ff, \go) \setminus \{0 \} =
\Theta(\fo, \gg) \setminus \{0 \}.$
\end{enumerate}
\begin{table}[ht]
\caption{Multiplicities of $0$ and $\pi/2$ angles between subspaces} \label{t2}
\renewcommand\arraystretch{1.5}
\noindent\[
\begin{array}{|c|c|c|}
\hline
{\text{Pair}}&{\theta=0}&{\theta=\pi/2}\\
\hline
{\Theta(\ff, \gg)}&{\dim \moo}&{\min \{ \dim \moi, \, \dim \mio \}}\\
\hline
{\Theta(\ff, \go)}&{\dim \moi}&{\min \{ \dim \moo, \, \dim \mii \}}\\
\hline
{\Theta(\fo, \gg)}&{\dim \mio}&{\min \{ \dim \moo, \, \dim \mii \}}\\
\hline
{\Theta(\fo, \go)}&{\dim \mii}&{\min \{ \dim \moi, \, \dim \mio \}}\\
\hline
\end{array}
\]
\end{table}
The multiplicities of the point angles $\theta$ in $\Theta(\ff,\gg)$
and $\Theta(\fo,\go)$ satisfying $0<\theta<{\pi}/{2}$
are the same, and equal to the multiplicities of point angles
$0 < {\pi}/{2} - \theta < {\pi}/{2}$ in $\Theta(\ff, \go)$ and $\Theta(\fo, \gg).$
\end{theorem}
\begin{proof}
Statement  (1) follows from Theorem \ref{7_relations} since 
\begin{eqnarray*}
\Theta(\ff, \gg) \setminus ( \{0\} \cup \{{\pi}/{2} \} ) &=& 
\hat{\Theta}(\ff, \gg) \setminus ( \{0\} \cup \{{\pi}/{2} \} ) \\
&=& \{ {\pi}/{2} - \hat{\Theta}(\ff, \go) \}
\setminus ( \{0\} \cup \{{\pi}/{2} \} ) \\
&=&
\{ {\pi}/{2} - \Theta(\ff, \go) \} \setminus ( \{0\} \cup \{{\pi}/{2} \} ),
\end{eqnarray*}
Using Theorem \ref{7_relations}(7) twice: first for $\ff$ and $\gg$,
next for $\gg$ and $\ff$, and then intersecting them gives (2).
Interchanging $\gg$ and $\go$ in (2) leads to (3).
The statements on multiplicities easily follow from Theorem \ref{7_relations} as
the entries in Table~\ref{t2} are just the minima between pairs of the
corresponding entries in Table \ref{t1}.
\end{proof}
\begin{remark}
Theorem \ref{7_relations}(1) allows us to introduce an equivalent sine-based definition:
\[
\hat{\Theta}(\ff, \gg)=\{\theta: \enspace \theta=\arcsin(\mu), \enspace \mu \geq 0,
\enspace \mum \in \Sigma((\pf \pgo)|_{\ff}) \} \subseteq [0,{\pi}/{2}].
\]
\end{remark}
\begin{remark}
Theorem \ref{7_relations}(2) implies
$
\Theta(\ff, \gg) \setminus \{ {\pi}/{2} \}= \hat{\Theta}(\ff, \gg)
\setminus \{ {\pi}/{2} \} =\hat{\Theta}(\gg, \ff) \setminus \{ {\pi}/{2} \}.
$
\end{remark}
\begin{remark}
We have 
$
\overline{\Theta(\ff, \gg) \setminus ( \{0\} \cup \{{\pi}/{2} \})}= \Theta(P_{\mm}\ff,P_{\mm}\gg),
%\setminus ( \{0\} \cup \{\frac{\pi}{2} \} )
$
in other words,
the projections $P_{\mm}\ff=\mf$ and $P_{\mm}\gg=\mg$
of the initial subspaces $\ff$ and $\gg$ onto their ``fifth part'' $\mm$ are in
generic position within $\mm$, see \cite{halmos69}, so
the zero and right angles can not belong to the set of
point angles $\Theta_p(P_{\mm}\ff,P_{\mm}\gg)$, but apart from
$0$ and $\pi/2$ the angles $\Theta(\ff,\gg)$ and $\Theta(P_{\mm}\ff,P_{\mm}\gg)$
are the same.
\end{remark}
\begin{remark}
Tables \ref{t1} and \ref{t2} give the absolute values of
the multiplicities of $0$ and $\pi/2$. If we need relative
multiplicities, e.g.,\ how many ``extra''  $0$ and
$\pi/2$ values are in $\Theta(\fo,\go)$ compared to
$\Theta(\ff,\gg)$, we can easily find the answers from Tables
\ref{t1} and \ref{t2} by subtraction, assuming that we subtract
finite numbers, and use identities such as $\dim \moo - \dim \mii
= \dim \ff - \dim \go$ and $\dim \moi - \dim \mio = \dim \ff -
\dim \gg$. Indeed, for the particular question asked above, we
observe that the multiplicity of $\pi/2$ is the same in
$\Theta(\fo,\go)$ and in $\Theta(\ff,\gg)$, but the difference in
the multiplicities of $0$ in $\Theta(\fo,\go)$ compared to in
$\Theta(\ff,\gg)$ is equal to $\dim \mii - \dim \moo = \dim \go -
\dim \ff$, provided that the terms that participate in the
subtractions are finite. Some comparisons require both the
dimension and the codimension of a subspace to be finite, thus,
effectively requiring $\dim \hh < \infty.$
\end{remark}
\subsection{Known Quantities as Functions of Angles}\label{known_quant}
The gap bounds the perturbation
of a closed linear operator by measuring the change
in its graph, while the minimum gap
between two subspaces determines if the sum of the subspaces is closed.
We connect the gap and the minimum gap
to the largest and to the nontrivial smallest principal angles.
E.g., for subspaces $\ff$ and $\gg$ in generic position,
i.e.,\ if $\mm=\hh$, we show that the gap and the minimum gap are the
supremum and the infimum, correspondingly, of the sine of the set of
angles between $\ff$ and $\gg$.

The gap (aperture) between subspaces $\ff$ and $\gg$ defined as,
e.g.,\ \cite{kato},
\[\gap(\ff,\gg)=\left\| \pf-\pg \right\| = 
\max \left\{\left\|\pf \pgo\right\|, \left\|\pg \pfo \right\|\right\}\]
is used to measure the distance between subspaces.
We now describe the gap in terms of the angles.
\begin{theorem}\label{gap_angles}
$
\min \left\{\min\left\{\cos^2(\hat{\Theta}(\ff, \gg))\right\} , 
\min \left\{\cos^2(\hat{\Theta}(\gg, \ff))\right\}\right \}=
1-\gap^2(\ff,\gg).
$
\end{theorem}
\begin{proof}
Let us consider both norms in the definition of the gap separately.
Using Theorem \ref{7_relations}, we have
\begin{eqnarray*}%\label{gaps_parts_1}
\|\pf \pgo\|^2&=&\sup_{\substack{u \in \hh  \\ \|u\|=1}} \| \pf \pgo u \|^2 =
\sup_{\substack{u \in \hh  \\ \|u\|=1}} (\pf \pgo u, \pf \pgo u)
\nonumber \\
&=&\sup_{\substack{u \in \hh  \\ \|u\|=1}} (\pgo \pf \pgo u, u)=
\|(\pgo\pf)|_{\go} \| = \max \{ \cos^2(\hat{\Theta}(\go, \ff))\}
\nonumber \\
&=&\max \{ \sin^2(\hat{\Theta}(\gg, \ff))\}=1-\min\{ \cos^2(\hat{\Theta}(\gg, \ff))\}.
\end{eqnarray*}
Similarly,
$
\|\pg \pfo\|^2= \max \{ \cos^2(\hat{\Theta}(\fo, \gg))\}=1-\min\{ \cos^2(\hat{\Theta}(\ff, \gg))\}.
$
\end{proof}

It follows directly from the above proof and the previous section that
\begin{cor}\label{cor:gap}
If $\gap(\ff,\gg)<1$ or if the subspaces are in generic position
then both terms under the minimum are the same and so
$\gap(\ff,\gg) = \max\{ \sin({\Theta}(\ff, \gg))\}.$
\end{cor}

Let
$
c(\ff, \gg)=\sup\{|(f,g)|: f\in\ff\ominus(\ff\cap\gg),\|f\| \leq 1,
g\in\gg\ominus(\ff\cap\gg),\|g\|\leq 1\},
$
as in \cite{deutsch}, which is a definition of the cosine of the \emph{angle of Friedrichs}.  
\begin{theorem}\label{angles_deutsch}
In terms of the angles,
$
c(\ff, \gg)=\cos \left( \inf \left\{\Theta(\ff, \gg) \setminus \{0\}\right\}\right).
$
\end{theorem}
\begin{proof}
Replacing the vectors $f=\pf u$ and $g=\pg v$
in the definition of $c(\ff, \gg)$ with the vectors $u$ and $v$
and using the standard equality of induced norms
of an operator and the corresponding bilinear form, 
%e.g.,\ \cite[Theorem 5.35, p.120]{weidmann}, 
we get
\begin{equation*}
c(\ff, \gg)= \sup_{\substack{u \in \hh \ominus \moo \\ \|u\|=1}}
\sup_{\substack{v \in \hh \ominus \moo \\ \|v\|=1}} |( u,\fg v)| =
\| (\fg)|_{\hh \ominus \moo}\|.
\end{equation*}
Using the five-parts decomposition,
$\fg=I_{\moo}\oplus 0_{\moi}\oplus 0_{\mio}\oplus 0_{\mii}\oplus (\fg)|_{\mm},$
thus ``subtracting'' the subspace $\moo$ from the domain of $\fg$
excludes $1$ from the point spectrum of $\fg$, and, thus,
$0$ from the set of point angles from $\ff$ to $\gg$ and,
by Theorem \ref{7_relations}(2),
from the set of point angles between $\ff$ and $\gg$.
\end{proof}

Let the \emph{minimum gap}, see \cite[\S~IV.4]{kato}, be defined as 
\begin{eqnarray*}
\gamma(\ff, \gg)=\inf_{\substack{f \in \ff, \, f \notin \gg}}
\frac{\dist(f,\, \gg)}{\dist(f,\, \ff \cap \gg)}.
\end{eqnarray*}
\begin{theorem}\label{angles_kato}
In terms of the angles,
$
\gamma(\ff, \gg)= \sin \left(\inf \left\{\Theta(\ff, \gg) \setminus \{0\}
\right\} \right).
$
\end{theorem}
\begin{proof}
We have $f \in \ff$ and $f \notin \gg$, so we can represent $f$ in the form $f=f_1+f_2$,
where $f_1 \in \ff\ominus(\ff \cap \gg)$, $f_1 \neq 0$ and $f_2 \in \ff \cap \gg$.
Then
\begin{eqnarray*}
\gamma(\ff, \gg)
&=&\inf_{\substack{f \in \ff, \, f \notin \gg}}
\frac{\dist(f,\, \gg)}{\dist(f,\, \ff \cap \gg)}\\
&=&\inf_{\substack{f_1 \in \ff\ominus(\ff \cap \gg), \, f_2 \in \ff \cap \gg}}
\frac{\|f_1+f_2-P_\gg f_1 - P_\gg f_2\|}{\|f_1+f_2-P_{\ff \cap \gg}f_1 - P_{\ff \cap \gg}f_2 \|}\\
&=& \inf_{\substack{f_1 \in \ff\ominus(\ff \cap \gg)}}
\frac{\|f_1-P_\gg f_1\|}{\|f_1-P_{\ff \cap \gg}f_1 \|}\\
&=& \inf_{\substack{f \in \ff\ominus(\ff \cap \gg)}}
\frac{\|f-P_\gg f\|}{\|f-P_{\ff \cap \gg}f \|}.
\end{eqnarray*}
But $f \in (\ff \cap \gg)^\perp$ and $\|f-P_{\ff \cap \gg} f\|=\|f\|$.
Since $\|\kappa f-P_\gg (\kappa f)\|=|\kappa|\|f-P_\gg f\|$, using the Pythagorean theorem we have
\begin{eqnarray*}
\gamma^2(\ff, \gg)
&=& \inf_{\substack{f \in \ff\ominus(\ff \cap \gg),}}
\frac{\|f-P_\gg f\|^2}{\|f\|^2}\\
&=& \inf_{\substack{f \in \ff\ominus(\ff \cap \gg), \, \|f\|=1}} \|f-P_\gg f\|^2\\
&=& \inf_{\substack{f \in \ff\ominus(\ff \cap \gg), \, \|f\|=1}} {1-\|P_\gg f\|^2}.
\end{eqnarray*}
Using the equality $\|P_\gg f\|=\sup_{\substack{g \in \gg, \, \|g\|=1}} |(f,g)|$ we get
\begin{eqnarray*}
\gamma^2(\ff, \gg)&=&{1-\sup_{\substack{f \in \ff\ominus(\ff \cap \gg),
\, g \in \gg, \, \|f\|=\|g\|=1}} |(f,g)|^2}\\
&=& {1-(c(\ff, \gg))^2}%\\
%&=& {1-( \sup \big \{ \cos \big(\Theta(\ff, \gg) \setminus \{0\} \big) \big \} )^2}\\
%&=& {\inf \big \{ 1- \cos^2 \big(\Theta(\ff, \gg) \setminus \{0\} \big) \big \}}\\
%&=& \inf \big \{ \sin^2 \big(\Theta(\ff, \gg) \setminus \{0\} \big))\big \}.
\end{eqnarray*}
and finally we use Theorem \ref{angles_deutsch}.
\end{proof}

Let us note that removing $0$ from the set of angles in Theorems
\ref{angles_deutsch} and \ref{angles_kato} changes the result after
taking the $\inf$, only if $0$ is present as an isolated
value in the set of angles, e.g.,\ it has no effect
for a pair of subspaces in generic position.

\subsection{The Spectra of Sum and Difference of Orthogonal Projectors}\label{sum_difference}
Sums and differences of a pair of orthogonal projectors often appear in applications.
Here, we describe their spectra in terms of the angles between the ranges of the projectors,
which provides a geometrically intuitive and uniform framework to analyze
the sums and differences of orthogonal projectors.
First, we connect the spectra of the product and of the difference of two orthogonal projectors.
\begin{lemma}\label{pq_p-q} (\cite[Theorem 1]{omlad}, \cite[Lemma 2.4]{kolRak}).
For proper subspaces $\ff$ and $\gg$ we have
$\Sigma(\fg)=\Si(\fgf)\subseteq[0,1]$ and
\[
\Sigma(\pg - \pf) \setminus (\{ -1 \} \cup \{ 0 \} \cup \{ 1 \}) =
\{ \pm (1-\sis)^{1/2}: \enspace \sis \in \Sigma(\fg) \setminus (\{0\} \cup \{ 1 \}) \}.
\]
\end{lemma}
Using Lemma \ref{pq_p-q}, we now characterize the spectrum of the
differences of two orthogonal projectors in terms of the angles between the corresponding subspaces.
\begin{theorem}\label{p-q_sin_f_g}
The multiplicity of the eigenvalue $1$ in $\Si(\pg - \pf)$ is equal to $\dim \mio$, the multiplicity of
the eigenvalue $-1$ is equal to $\dim \moi$, and the multiplicity of the eigenvalue $0$ is equal to
$\dim \moo + \dim \mii$, where $\moo$, $\moi$, $\mio$ and $\mii$
are defined in \S~\ref{prelims}.
For the rest of the spectrum, we have the following:
\[\Si(\pf - \pg) \setminus (\{ -1 \} \cup \{ 0 \} \cup \{ 1 \})=
\pm \sin (\Theta(\ff, \gg)) \setminus (\{ -1 \} \cup \{ 0 \} \cup \{ 1 \}).\]
\end{theorem}
\begin{proof}
The last statement follows from Lemma \ref{pq_p-q} and Definition
\ref{angles_from_cos}. To obtain the results concerning the multiplicity of eigenvalues $1$,  $-1$
and $0$, it suffices to use the decomposition of these
projectors into five parts, given in \S~\ref{prelims}.
\end{proof}

In some applications, e.g.,\ in domain decomposition methods,
see \S \ref{DDM}, the
distribution of the spectrum of the sum of projectors is important. 
We directly reformulate \cite[Corollary 4.9, p. 86]{bjorMand}, see also \cite[p. 298]{vidav}, 
in terms of the angles between subspaces:
\begin{theorem}\label{sum_of_projs_Jan}
For any nontrivial pair of orthogonal projectors $\pf$ and $\pg$ on $\hh$ the spectrum
of the sum $\pf+\pg$, with the possible exception of the point $0$, lies in the closed interval
of the real line $[1-\| \fg \|,  1+ \| \fg \| ]$, and the following identity holds:
\[
\Si(\pf + \pg) \setminus (\{0\} \cup \{ 1 \})=\{ 1 \pm \cos(\Theta(\ff,\gg)) \}
\setminus (\{0\} \cup \{ 1 \}).
\]
\end{theorem}
%%%%%%%%%%%%%%%%%%%%%%%%%%%%%%%%%%%%%%%%%%%%%%%%%%%%%%%%%%%%%%%%%%%%
\section{Principal Vectors, Subspaces and Invariant Subspaces}\label{princ_vectors_1}
In this section,  we basically follow \citet[Section 2.8]{juju_thesis} to 
introduce principal invariant subspaces for a pair of subspaces 
by analogy with invariant subspaces of operators.  
Given the principal invariant subspaces (see Definition \ref{princ_invar_subsp} below) 
of a pair of subspaces $\ff$ and $\gg$, 
we construct the principal invariant subspaces 
for pairs $\ff$ and $\go$, $\fo$ and $\gg$, $\fo$ and $\go$.
We describe relations between orthogonal projectors onto
principal invariant subspaces. We show that, in particular cases,
principal subspaces and principal vectors can be defined
essentially as in the finite dimensional case,
and we investigate their properties. Principal vectors, subspaces and
principal invariant subspaces reveal the fine structure of the
mutual position of a pair of subspaces in a Hilbert space.
Except for Theorem \ref{thm.myinvariant}, all other statements can be found in 
\cite[sections 2.6-2.9]{juju_thesis}, which we refer the reader to for detailed proofs
and more facts.  
%%%%%%%%%%%%%%%%%%%%%%%%%%%%%%%%%%%%%%%%%%%%%%%%%%%%%%%%%%%%%%%%%
\subsection{Principal Invariant Subspaces}\label{princ_vectors_2}
Principal invariant subspaces for a pair of subspaces generalize 
the already known notion of principal vectors, e.g.,\ \cite{w83}. 
We give a geometrically intuitive definition of
principal invariant subspaces and connect them with invariant subspaces of
the product of the orthogonal projectors. 
\begin{definition}\label{princ_invar_subsp}
A pair of subspaces $\uu \subseteq \ff$ and $\vv \subseteq \gg$ is called a pair of
principal invariant subspaces for the subspaces $\ff$ and $\gg$, if
$\pf\vv\subseteq\uu$ and $\pg\uu\subseteq\vv.$
We call the pair $\uu\subseteq\ff$ and $\vv\subseteq\gg$
nondegenerate if $\overline{\pf\vv}=\uu\neq\{0\}$ and $\overline{\pg\uu}=\vv\neq\{0\}$
and strictly nondegenerate if $\pf\vv=\uu\neq\{0\}$ and $\pg\uu=\vv\neq\{0\}.$
\end{definition}
This definition is different from that used in 
\cite[Section 2.8, p. 57]{juju_thesis}, where only what we call here 
strictly nondegenerate principal invariant subspaces are defined. 

The following simple theorem deals with enclosed principal invariant subspaces.
\begin{theorem}\label{angl_sub_sub_1}
Let $\uu \subset \ff$ and $\vv \subset \gg$ be a pair of principal invariant
subspaces for subspaces $\ff$ and $\gg$, and $\underline{\uu} \subset \uu$,
$\underline{\vv} \subset \vv$ be a pair of principal invariant  subspaces for subspaces
$\uu$ and $\vv$. Then $\underline{\uu}$, $\underline{\vv}$ form a pair of principal invariant
subspaces for the subspaces $\ff$, $\gg$, and
$\Theta(\underline{\uu},\underline{\vv}) \subseteq \Theta(\uu,\vv) \subseteq \Theta(\ff,\gg).$
\end{theorem}
Definition \ref{princ_invar_subsp} resembles the notion of 
invariant subspaces. The next theorem completely clarifies this connection 
for general principal invariant subspaces. 
\begin{theorem}\label{thm.myinvariant}
The subspaces $\uu\subseteq\ff$ and $\vv\subseteq\gg$ form a 
pair of principal invariant subspaces for the subspaces $\ff$ and $\gg$ 
if and only if $\uu\subseteq\ff$ is an invariant subspace
of the operator $(\fg)|_\ff$ 
and $\vv=\overline{\pg\uu}\oplus\vv_0$, where $\vv_0\subseteq\mio=\gg\cap\ff^\perp$.
\end{theorem}
\begin{proof}
Conditions $\pf\vv\subseteq\uu$ and $\pg\uu\subseteq\vv$ imply 
$\pf\pg\uu\subseteq\pf\vv\subseteq\uu.$ Let us consider 
$v_0\in\vv\ominus\overline{\pg\uu}=\vv\cap\uu^\perp$ 
(the latter equality follows from $0=(v_0,\pg u)=(v_0,u),\,\forall u\in\uu$).
We have $\pf v_0\in\uu^\perp$ since $\uu\subseteq\ff$, but our assumption
$\pf\vv\subseteq\uu$ assures that $\pf v_0\in\uu$, so $\pf v_0=0$, which means that 
$\vv_0\subseteq\mio,$ as required.

To prove the converse, let $\pf\pg\uu\subseteq\uu$ and 
$\vv=\overline{\pg\uu}\oplus\vv_0$. Then 
$\pf\vv=\pf\overline{\pg\uu}\subseteq\uu$ since $\uu$ is closed. 
$\pg\uu\subseteq\vv$ follows from the formula for $\vv.$
\end{proof}

If the subspace $\mio$ is trivial, the 
principal invariant subspace $\vv$ that corresponds to $\uu$ is clearly unique. 
The corresponding statement for $\uu$, given $\vv$, we get from 
Theorem \ref{thm.myinvariant} by swapping $\ff$ and $\gg.$
We now completely characterize (strictly) nondegenerate principal invariant subspaces
using the corresponding angles.
\begin{theorem}\label{2-inv-2}
The pair $\uu \subseteq \ff$ and $\enspace \vv \subseteq \gg$
of principal invariant subspaces for the subspaces $\ff$ and $\gg$
is nondegenerate if and only if 
both operators $(\fg)|_\uu$ and $(\gf)|_\vv$ are invertible, 
i.e.,\  ${\pi}/{2}\notin\hat{\Theta}_p(\uu,\vv)\cup\hat{\Theta}_p(\vv,\uu),$
and strictly nondegenerate if and only if
each of the inverses is bounded, 
i.e.,\  ${\pi}/{2}\notin\hat{\Theta}(\uu,\vv)\cup\hat{\Theta}(\vv,\uu),$
or equivalently in terms of the gap, $\gap(\uu,\vv)=\|\pu-\pv\|<1$.
\end{theorem}
\begin{proof}
We prove the claim for the operator $(\fg)|_\uu,$ and the claim for the other
operator follows by symmetry. 
Definition \ref{princ_invar_subsp} uses  
$\overline{\pf\vv}=\uu\neq\{0\}$ for nondegenerate  principal invariant subspaces. 
 At the same time, Theorem \ref{thm.myinvariant}
holds, so $\vv=\overline{\pg\uu}\oplus\vv_0$, where $\vv_0\subseteq\mio=\gg\cap\ff^\perp$.
So $\uu=\overline{\pf\vv}=\overline{\fg\uu}.$ Also by Theorem \ref{thm.myinvariant},
$\uu\subseteq\ff$ is an invariant subspace of the operator $(\fg)|_\ff$, so
$\uu=\overline{\fg\uu}=\overline{(\fg)|_\uu\uu}$. 
Since $(\fg)|_\uu$ is Hermitian,
its null-space is trivial (as the orthogonal in $\uu$ complement to its range 
which is dense in $\uu$),
i.e.,\ the operator  $(\fg)|_\uu$ is one-to-one and thus invertible. 
For strictly nondegenerate principal invariant subspaces, 
${(\fg)|_\uu\uu}=\uu$, so the operator $(\fg)|_\uu$
by the open mapping theorem has a continuous and thus bounded inverse.

Conversely, by Theorem \ref{thm.myinvariant} 
$\uu\subseteq\ff$ is an invariant subspace of the operator $(\fg)|_\ff$, 
so the restriction $(\fg)|_\uu$ is correctly defined. 
The operator $(\fg)|_\uu$ is invertible by assumption, thus  
its null-space is trivial, and so its range is dense: 
$\uu=\overline{(\fg)|_\uu\uu}=\overline{\fg\uu}$.
By Theorem \ref{thm.myinvariant}, 
$\vv=\overline{\pg\uu}\oplus\vv_0$, therefore  
$\overline{\pf\vv}=\overline{\fg\uu}=\uu.$
The other equality, $\overline{\pg\uu}=\vv\neq\{0\}$,
of Definition \ref{princ_invar_subsp} for nondegenerate  principal invariant subspaces,
is proved similarly using the assumption that $(\gf)|_\gg$ is invertible. 
If, in addition, each of the inverses is bounded, the corresponding ranges are closed, 
$\uu=\fg\uu$ and $\vv=\fg\vv$ and we obtain 
$\pf\vv=\uu\neq\{0\}$ and $\pg\uu=\vv\neq\{0\}$
as is needed in Definition \ref{princ_invar_subsp} for strictly nondegenerate  principal invariant subspaces.

The equivalent formulations of conditions of the theorem in terms of the angles and the gap  
follow directly from Definitions \ref{angles_from_cos} and \ref{multiplicity_of_angle} 
and Theorem \ref{gap_angles}. 
\end{proof}

Theorem \ref{andrew_ilya} introduces 
the unitary operator $W$ that gives the unitary equivalence of 
$\fgf$ and $\gfg$ and, if $\gap(\ff,\gg)<1$, the unitary equivalence by \eqref{dav_ka_ka} of 
$\pf$ and $\pg$. Now we state that the same $W$ makes orthogonal projectors
$\pu$ and $\pv$  unitarily equivalent for strictly nondegenerate  principal invariant subspaces
$\uu \subset \ff$ and $\vv \subset \gg$, and 
we obtain expressions for the orthogonal projectors.
\begin{theorem}\label{inv-proj}
Let  $\uu\subseteq\ff$ and $\vv\subseteq\gg$
be a pair of  strictly nondegenerate principal invariant subspaces for the subspaces $\ff$ and $\gg$, 
and $W$ be defined as in Theorem \ref{andrew_ilya}.
Then $\vv=W \uu$ and $\uu=W^* \vv,$ 
while the orthoprojectors satisfy 
$\pv=W \pu W^*=\pg\pu((\fg)|_\uu)^{-1}\pu\pg$ 
and $\pu=W^* \pv W=\pf\pv((\gf)|_\vv)^{-1}\pv\pf.$ 
\end{theorem}
%\begin{proof}
%We prove the claim for $\pv$ and the claim for $\pu$ follows by symmetry. Let
%$P=\pg\pu((\fg)|_\uu)^{-1}\pu\pg,$
%where by Theorem \ref{2-inv-2} the operator $(\fg)|_\uu$ has a bounded inverse.
%If $v \in \vv$ then $v= \pg u,$ for some $u \in \uu$ and
%\begin{eqnarray*}
%P v &=& \pg\pu((\fg)|_\uu)^{-1}\pu\pg u = \pg\pu((\fg)|_\uu)^{-1}\pu\pf\pg u\\
%&=& \pg\pu((\fg)|_\uu)^{-1}\pf\pg u = \pg\pu u =\pg u=v.
%\end{eqnarray*}
%Now we show that $P\vo=0, \forall\vo\in\vvo$.
%Let $h\in\hh$. Then $(P\vo,h)=(\vo,Ph)=0$ since
%$Ph=\pg\pu((\fg)|_\uu)^{-1}\pu\pg h\in\vv$, and this implies
%$P \vo =0$. $P$ is Hermitian, thus $\pv=P=\pg\pu((\fg)|_\uu)^{-1}\pu\pg.$
%\end{proof}
The proof of Theorem \ref{inv-proj} is straightforward and can be found in 
\cite[\S 2.8]{juju_thesis}. \citet[\S 2.9]{juju_thesis} also develops 
the theory of principal invariant subspaces, using the spectral decompositions, 
e.g., below is \cite[Theorem 2.108]{juju_thesis}: 
\begin{theorem}\label{thm3_17} 
Let $\{E_1\}$ and $\{E_2\}$ be spectral measures  
of the operators $(\fg)|_\ff$ and $(\gf)|_\gg$, respectively. Let 
$\Theta\subseteq{\Theta}(\uu,\vv)\setminus\{{\pi}/{2}\}$ 
be a closed Borel set, and define $P_{\uu(\Theta)}=\int_{\cos(\Theta)} dE_1(\lambda)$ and 
$P_{\vv(\Theta)}=\int_{{\cos(\Theta)}} dE_2(\lambda)$.
Then $\uu(\Theta)\subset\ff$ and $\vv(\Theta)\subset\gg$ is a pair of strictly nondegenerate 
principal invariant subspaces and
\begin{equation*}
P_{\vv({\Theta})} =\pg \left\{\int_{{\cos(\Theta)}} \frac{1}{\lambda} dE_1(\lambda) \right\} \pg, 
\end{equation*}
and $\Theta=\hat{\Theta}(\uu(\Theta),\vv(\Theta))=\hat{\Theta}(\vv(\Theta),\uu(\Theta)).$ 
\end{theorem}
\begin{proof}
We have 
$\int_{{\cos(\Theta)}} \frac{1}{\lambda} dE_1(\lambda)=
\left((\pf\pu)|_{\uu}\right)^{-1}
=\pu\left((\pf\pg)|_{\uu}\right)^{-1}\pu$
(where we denote $\uu=\uu(\Theta)$), which we plug into the 
expression for the orthogonal projector $\pv$ of Theorem \ref{inv-proj}. 
\end{proof}
For a pair of principal invariant subspaces $\uu \subset \ff$ and $\vv \subset \gg$,
using Theorems \ref{thm.myinvariant} and \ref{2-inv-2} 
we define the corresponding principal invariant subspaces in $\fo$ and $\go$ as
$\luo=\pfo\vv$ and $\lvo=\pgo\uu$, and describe their properties in the next theorem.  
\begin{theorem}\label{different_pairs_princ_invar_subsp}
Let $\uu$ and $\vv$ be a pair of principal invariant subspaces for subspaces $\ff$ and $\gg$
and $0, {\pi}/{2} \notin \hat{\Theta}(\uu, \vv) \cup \hat{\Theta}(\vv, \uu).$
Then $\luo=\pfo\vv$ and $\lvo=\pgo\uu$ are closed and
\begin{itemize}
 \item  $\uu$ and $\vv$ is a pair of  strictly nondegenerate  principal invariant subspaces for subspaces $\ff$ and $\gg$;
 \item  $\luo$ and $\vv$ is a pair of  strictly nondegenerate  principal invariant subspaces for subspaces $\fo$ and $\gg$
        and $P_{\luo}$ and $P_{\vv}$ are unitarily equivalent;
 \item  $\uu$ and $\lvo$ is a pair of  strictly nondegenerate principal invariant subspaces for subspaces $\ff$ and $\go$
        and $P_{\uu}$ and $P_{\lvo}$ are unitarily equivalent;
 \item  $\luo$ and $\lvo$ is a pair of  strictly nondegenerate principal invariant subspaces for subspaces $\fo$ and $\go$
        and $P_{\luo}$ and $P_{\lvo}$ are unitarily equivalent.
\end{itemize}
\end{theorem}
\begin{proof}
The statements follow directly from Theorems \ref{thm.myinvariant} and \ref{2-inv-2} 
applied to the corresponding pairs of subspaces. 
The closedness of $\luo$ and $\lvo$ can be alternatively derived from 
Theorem \ref{angles_deutsch} and \cite[Theorem 22]{deutsch}.
\end{proof}
%%%%%%%%%%%%%%%%%%%%%%%%%%%%%%%%%%%%%%%%%%%%%%%
\subsection{Principal Subspaces and Principal Vectors}\label{princ_vectors_3}
For a pair of principal invariant subspaces 
$\uu\subset\ff$ and $\vv\subset\gg$, if  
the spectrum $\Si((\fg)|_\uu)$ consists of one number,
which belongs to $(0,1]$ and which
we denote by $\cos^2(\theta)$, we can use Theorem \ref {inv-proj}
to define a pair of principal subspaces
corresponding to an angle $\theta$:
\begin{definition}\label{principal_subspaces}
Let $\theta\in\Theta(\ff,\gg)\setminus\{{\pi}/{2}\}$.  
Nontrivial subspaces $\uu\subseteq \ff$ and $\vv\subseteq\gg$ define a pair of
principal subspaces for subspaces $\ff$ and $\gg$ corresponding to the angle $\theta$ if
$(\pf \pv)|_\ff={\cos}^2(\theta) \pu$ and $(\pg \pu)|_\gg= {\cos}^2(\theta) \pv.$
Normalized vectors $u=u(\theta)\in\ff$ and $v=v(\theta)\in\gg$ form a pair
of principal vectors for subspaces $\ff$ and $\gg$ corresponding to the angle
$\theta$ if $\pf v=\cos(\theta) u$ and $\pg u= \cos(\theta) v.$
\end{definition}
We exclude $\theta=\pi/2$ in Definition \ref{principal_subspaces} so that 
principal subspaces belong to the class of strictly nondegenerate principal invariant subspaces. 
We describe the main properties of principal subspaces and principal vectors
that can be checked directly (for details, see \cite{juju_thesis}).
The first property characterizes principal subspaces as
eigenspaces of the products of the corresponding projectors.
\begin{theorem}\label{principal_subsp_eigenspaces}
Subspaces $\uu\subset\ff$ and $\vv\subset\gg$
form a pair of principal subspaces for
subspaces $\ff$ and $\gg$ corresponding to the angle
$\theta \in \Theta(\ff, \gg) \setminus \{{\pi}/{2} \}$
if and only if $\theta \in \Theta_p(\ff, \gg) \setminus \{ {\pi}/{2} \}$
and $\uu$ and $\vv$ are the eigenspaces of the operators
$(\fg)|_\ff$ and $(\gf)|_\gg$, respectively, corresponding
to the eigenvalue ${\cos}^2(\theta)$.
In such a case, $\Theta(\uu, \vv)=\Theta_p(\uu, \vv)=\{\theta \}$.
All pairs of principal vectors $u$ and $v$
of subspaces $\ff$ and $\gg$ corresponding to the angle
$\theta$ generate the largest principal subspaces $\uu$ and $\vv$ 
corresponding to the angle $\theta$.  
\end{theorem}
\begin{theorem}\label{principal_subsp_properties}
Let $\uu(\theta),\,\uu(\phi) \subset \ff$,
and $\vv(\theta),\,\vv(\phi) \subset \gg$
be the principal subspaces for subspaces $\ff$ and $\gg$ corresponding to the
angles $\theta,  \phi \in \Theta_p(\ff, \gg) \setminus \{{\pi}/{2} \} $. Then 
$P_{\uu(\theta)} P_{\uu(\phi)}=P_{\uu(\theta) \cap \, \uu(\phi)}$; 
       $P_{\vv(\theta)} P_{\vv(\phi)}=P_{\vv(\theta) \cap \vv(\phi)}$ ;
$P_{\uu(\theta)}$ and $P_{\vv(\phi)}$ are mutually orthogonal if
       $\theta \neq \phi$ (if $\theta = \phi$ we can
       choose $\vv(\theta)$ such that $P_{\uu(\theta)} P_{\vv(\theta)}=P_{\uu(\theta)} \pg$);
for given $\uu(\theta)$ we can choose $\vv(\theta)$ such that
       $P_{\vv(\theta)} P_{\uu(\theta)}=P_{\vv(\theta)} \pf$.
\end{theorem}
\begin{cor}\label{diff_prin_sub}[of Theorem \ref{different_pairs_princ_invar_subsp}]
Let $\uu$ and $\vv$ be the principal subspaces for subspaces $\ff$ and $\gg$, corresponding to
the angle $\theta \in \Theta_p(\ff, \gg) \setminus ( \{0\} \cup \{{\pi}/{2} \} )$. 
Then $\luo=\pfo\vv$  and $\lvo=\pgo\uu$ are closed and
\begin{itemize}
 \item  $\luo$, $\vv$ are the principal subspaces for subspaces $\fo$ and $\gg$, corresponding
 to the angle ${\pi}/{2} - \theta$;
 \item  $\uu$, $\lvo$ are the principal subspaces for subspaces $\ff$ and $\go$, corresponding
 to the angle ${\pi}/{2} - \theta$;
 \item  $\luo$, $\lvo$ are the principal subspaces for subspaces $\fo$ and $\go$, corresponding
 to the angle $\theta$.
\end{itemize}
Let $u$ and $v$ form a pair of principal vectors for the subspaces $\ff$ and $\gg$,
corresponding to the angle $\theta$.
Then $\ulo = (v- \cos(\theta) u)/ \sin(\theta)$ and $\vlo =(u - \cos(\theta) v)/ \sin(\theta)$
together with $u$ and $v$ describe the pairs of principal vectors. 
\end{cor}
%%%%%%%%%%%%%%%%%%%%%%%%%%%%%%%%%%%%%%%%%%%%%%%%%%%%%%%%%%%%%%%%%%%%%%%%%%%%%%%%%%%%
\section{Bounding the Changes in the Angles}\label{prox_angl}
Here we prove bounds on the change in the (squared cosines
of the) angles from one subspace to another where the subspaces change.
These bounds allow one to estimate the sensitivity of the
angles with respect to the changes in the subspaces.
For the finite dimensional case, such bounds are known, e.g., \cite{ka02,ka03}.
To measure the distance between two bounded real sets $S_1$ and $S_2$
we use the Hausdorff distance, e.g., \cite{kato},
$\dist(S_1,S_2)=\max\{\sup_{u\in S_1}\dist(u,S_2),\sup_{v\in S_2}\dist(v,S_1) \},$
where $\dist(u,S)= \inf_{v\in S}|u-v|$ is the distance from the point $u$ to the set $S$.
The following theorem estimates the proximity of the set of squares
of cosines of $\hat{\Theta}(\ff, \gg)$ and the set of squares of cosines of
$\hat{\Theta}(\ff, \ggt)$, where $\ff$, $\gg$ and $\ggt$ are nontrivial subspaces of $\hh$.
\begin{theorem}\label{cosines_squared}
$
\dist(\cos^2(\hat{\Theta}(\ff,\gg)), \cos^2(\hat{\Theta}(\ff,\ggt))) \leq \gap(\gg, \ggt).
$
\end{theorem}
\begin{proof}
$\cos^2(\hat{\Theta}(\ff,\gg))=\Sigma((\fg)|_{\ff})$
and $\cos^2(\hat{\Theta}(\ff,\ggt))=\Sigma((\pf\pgt)|_{\ff})$ by Definition \ref{angles_from_cos}. Both
operators $(\fg)|_\ff$ and $(\pf\pgt)|_{\ff}$ are selfadjoint.
By \cite[Theorem 4.10, p. 291]{kato}, 
\[\dist(\Sigma((\fg)|_{\ff}), \Sigma((\pf\pgt)|_{\ff}))\leq\|(\fg)|_{\ff}-(\pf\pgt)|_{\ff}\|.\]
Then,
$\|(\fg)|_{\ff}-(\pf\pgt)|_{\ff}\|\leq\|\pf\|\|\pg-\pgt\|\|\pf\|\leq\gap(\gg, \ggt).$
\end{proof}
The same result holds also if the first subspace, $\ff$, is changed in  $\hat{\Theta}(\ff,\gg))$:
\begin{theorem}\label{cosines_squared_first}
$
\dist(\cos^2(\hat{\Theta}(\ff,\gg)), \cos^2(\hat{\Theta}(\fft,\gg))) \leq \gap(\ff, \fft).
$
\end{theorem}
\begin{proof}
The statement of the theorem immediately follows from Theorem \ref{ritz_estim},
which is independently proved in the next section, where one takes %$\ff = \ff$, $\gg=\fft$ and
$A=\pg$.
\end{proof}
We conjecture that similar generalizations to the case of infinite dimensional subspaces 
can be made for bounds involving changes in the sines and cosines (without squares) 
of the angles extending known bounds \cite{ka02,ka03} for the finite dimensional case. 
%%%%%%%%%%%%%%%%%%%%%%%%%%%%%%%%%%%%%%%%%%%%%%%%%%%%%%%%%%%%%%%%%%%%%%%%%%%%%%%%%%%%
\section{Changes in the Ritz Values and Rayleigh-Ritz error bounds}\label{ritz_val}
Here we estimate how Ritz values of a
selfadjoint operator change with the change of a vector,
and then we extend this result to estimate the change of
Ritz values with the change of a (infinite dimensional) trial subspace, 
using the gap between subspaces, $\gap(\ff,\gg)=\|P_{\ff}-P_{\gg}\|.$ 
Such results are natural extensions of the results of the
previous section that bound the change in the squared
cosines or sines of the angles, since in the particular case where
the  selfadjoint operator is an orthogonal projector
its Ritz values are exactly the squared
cosines of the angles from the trial subspace of the
Rayleigh-Ritz method  to the range of the orthogonal projector.
In addition, we prove a spectrum error bound that 
characterizes the change in the Ritz values for an invariant subspace, 
and naturally involves the gap squared; see \cite{ko06,akpp06,ka07} for 
similar finite dimensional results. 

Let $A \in \bb(\hh)$ be a selfadjoint operator. Denote by
$\lambda(f)=(f,Af)/(f,f)$ the Rayleigh quotient of an operator $A$
at a vector $f\neq0$.
%Denote also
%$m=\inf_{f \neq 0} \lambda (f; \, A)$, $M=\sup_{f \neq 0} \lambda (f; \, A)$.
In the following lemma, we estimate changes in the Rayleigh quotient
with the change in a vector. This estimate has been
previously proven only for real finite dimensional spaces \cite{ka03}.
Here, we give a new proof that works both for real and
complex spaces.
\begin{lemma}\label{rayleigh_quotient}
Let $A \in \bb(\hh)$ be a selfadjoint operator on a Hilbert space $\hh$
and $f, g \in \hh$ with $f, g \neq 0$.
Then
\begin{equation} \label{rayleigh_perturb}
|\lambda (f) - \lambda (g)| \leq
(\max\{\Sigma(A)\} - \min\{\Sigma(A)\})  %(M-m)
\sin(\theta (f,g)).
\end{equation}
\end{lemma}
\begin{proof}
We use the so-called ``mini-dimensional'' analysis, e.g.,\
\cite{k,ks}.
Let $S=\Span\{f,g\} \subset \hh$ be a two dimensional subspace (if $f$ and $g$ are
linearly dependent then the Rayleigh quotients are the same and the assertion is trivial).
Denote $\tilde{A}=(P_SA)|_S$ and two eigenvalues of $\tilde{A}$ by
$\lambda_1 \leq \lambda_2$. By well known properties of the Rayleigh-Ritz method,
we have
$\lambda (f), \lambda (g) \in [\lambda_1,\lambda_2] \subseteq
[\max\{\Sigma(A)\}, \min\{\Sigma(A)\}]$.
In the nontrivial case  $\lambda (f) \neq \lambda (g)$, we then have
the strong inequality $\lambda_1 < \lambda_2$.

In this proof, we extend the notation of the Rayleigh quotient of an operator $A$
at a vector $f$ to $\lambda (f; \, A) = (f, Af)/(f,f)$ to explicitly include $A$.
It is easy to see that
$\lambda (f; \, A) = \lambda (f; \,\tilde A)$ and that the same holds for vector $g$.
Then, since $[\lambda_1,\lambda_2] \subseteq [\max\{\Sigma(A)\}, \min\{\Sigma(A)\}]$
the statement of the lemma would follow from the 2D estimate
$|\lambda(f;\,\tilde{A})-\lambda(g;\,\tilde{A})|\leq (\lambda_2-\lambda_1)\sin(\theta(f,g))$
that we now have to prove.
The latter estimate is clearly invariant with respect to a shift and scaling of $\tilde{A}$.
Let us use the transformation
$\bar{A}=(\tilde{A}-\lambda_1 I)/(\lambda_2-\lambda_1)$ then
the estimate we need to prove turns into
$|\lambda(f;\,\bar{A})-\lambda(g;\,\bar{A})|\leq \sin(\theta(f,g)),$
but the operator $\bar{A}$
has two eigenvalues, zero and one, and thus
is an orthoprojector on some one dimensional subspace $\Span \{h\} \subset S$.
Finally,
$\lambda(f;\,\bar{A})=(f, P_h f)/(f,f)=\cos^2(\theta(h,f))$ and, similarly, 
$\lambda(g;\,\bar{A})=(g, P_h g)/(g,g)=\cos^2(\theta(h,g))$. But
$|\cos^2(\theta(h,f))-\cos^2(\theta(h,g))|=
| \,\|P_hP_fP_h||-\|P_hP_gP_h|| \,| \leq \|P_f-P_g\|=\sin(\theta(f,g)).$
\end{proof}

In the Rayleigh-Ritz method for a selfadjoint operator $A \in \bb(\hh)$
on a trial subspace $\ff$ the spectrum  $\Si((P_{\ff}A)|_{\ff})$ is called the
set of Ritz values, corresponding to $A$ and $\ff$. The next result of this
section is an estimate of a change in the Ritz values, where
one trial subspace, $\ff$, is replaced with another, $\gg$.
For finite dimensional subspaces such a result is
obtained in \cite{ka03}, where the maximal distance between pairs of
individually ordered Ritz values is used to measure
the change in the Ritz values. Here, the trial subspaces may be
infinite dimensional, so the Ritz values may form rather general
sets on the real interval $[\min\{\Sigma(A)\},\max\{\Sigma(A)\}]$
and we are limited to the use of the Hausdorff distance
between the sets, which does not take into account the ordering and
multiplicities.
\begin{theorem}\label{ritz_estim}
Let $A \in \bb(\hh)$ be a selfadjoint operator and $\ff$ and $\gg$ be nontrivial subspaces
of $\hh$. Then a bound for the Hausdorff distance between the Ritz values of $A$, with respect to the
trial subspaces $\ff$ and $\gg$, is given by the following inequality
%\begin{equation}%\label{ritz_prox}
\[
\dist(\Si((P_{\ff}A)|_{\ff}), \,  \Si((P_{\gg}A)|_{\gg})) \leq
(\max\{\Sigma(A)\} - \min\{\Sigma(A)\})\,\gap(\ff,\gg).%(M-m)\|P_{\ff}-P_{\gg}\|.
\]
%\end{equation}
\end{theorem}
\begin{proof}
%(With contribution of I. Lashuk, UC Denver).
If $\gap(\ff,\gg)=1$ then the assertion holds since the both
spectra are subsets of $[\min\{\Sigma(A)\},\max\{\Sigma(A)\}]$.
Consequently we can assume
without loss of generality that $\gap(\ff,\gg)<1$. Then we
have $\gg=W \ff$ with $W$ defined by \eqref{dav_ka_ka}.
Operators $(P_{\gg}A)|_{\gg}$ and $\left(W^*(P_{\gg}A)|_{\gg}W\right)|_{\ff}$ are
unitarily equivalent, since $W$ is an isometry on $\ff$, therefore, their spectra are the same.
Operators $(P_{\ff}A)|_{\ff}$ and $\left(W^*(P_{\gg}A)|_{\gg}W\right)|_{\ff}$
are selfadjoint on the space $\ff$
and using \cite[Theorem 4.10, p. 291]{kato} we get
\begin{eqnarray}\nonumber
\dist(\Si((P_{\ff}A)|_{\ff}), \,  \Si((P_{\gg}A)|_{\gg}))
&=& \dist(\Si((P_{\ff}A)|_{\ff}), \,  \Si(\left(W^*(P_{\gg}A)|_{\gg}W\right)|_{\ff}))\\
&\leq& \|\left(P_{\ff}A-W^*(P_{\gg}A)|_{\gg}W\right)|_{\ff}\|.\label{ineq_77}
\end{eqnarray}
Then
\begin{eqnarray*}
\|\left(P_{\ff}A-W^*(P_{\gg}A)|_{\gg}W\right)|_{\ff}\|
&=& \sup_{\substack{\|f\|=1, \, f \in \ff}} |((P_{\ff}A-W^*(P_{\gg}A)|_{\gg}W) f,f)|\\
&=& \sup_{\substack{\|f\|=1, \, f \in \ff}} |(Af,f)-(AWf,Wf)|. 
\end{eqnarray*}
We have 
$| (f,Af)-(Wf,AWf)| \leq (\max\{\Sigma(A)\} - \min\{\Sigma(A)\})
\sqrt{1-|(f,Wf)|^2},$ $\forall f \in \ff$, $\|f\|=1$ by Lemma \ref{rayleigh_quotient}.
We need to estimate $|(f,Wf)|$ from below.
From the polar decomposition $P_{\gg} P_{\ff} = W \sqrt{P_{\ff} P_{\gg} P_{\ff}}$, 
we derive the equalities 
\[
(f,Wf)=(P_{\gg} P_{\ff} f, Wf)=(W^* P_{\gg} P_{\ff} f,f)=
(\sqrt{P_{\ff} P_{\gg}P_{\ff}}f,f),
\]
where we have
$
\sqrt{P_{\ff} P_{\gg} P_{\ff}}|_{\ff}=\sqrt{(P_{\ff} P_{\gg} P_{\ff})|_{\ff}}=
\sqrt{(P_{\ff} P_{\gg})|_{\ff}},
$
since $\ff$ is an invariant subspace of the operator $P_{\ff} P_{\gg} P_{\ff}$.
Thus,
$
(f,Wf) =(\sqrt{(P_{\ff} P_{\gg})|_{\ff}}f,f)
\geq \min\{\cos(\hat\Theta(\ff,\gg))\}
$
by Definition \ref{angles_from_cos}.
Finally, by assumption, $\gap(\ff,\gg)<1$, thus
Corollary~\ref{cor:gap} gives
$\min\{\cos^2(\hat\Theta(\ff,\gg))\}=1-\gap^2(\ff,\gg).$
\end{proof}

Finally, we assume that %one of the subspaces, e.g., 
$\ff$ is $A$-invariant, which 
implies that the set of the values $\Si((P_{\ff}A)|_{\ff})$ 
is a subset, namely $\Si(A|_{\ff})$, 
of the spectrum of $A.$ The change in the Ritz values, bounded in 
Theorem \ref{ritz_estim}, can now be interpreted as a spectrum error  
in the Rayleigh-Ritz method. The result of Theorem \ref{ritz_estim} here is improved 
since the new bound involves the gap squared as in \cite{akpp06,ka07}.  
\begin{theorem}\label{ritz_estim_sq}
Under the assumptions of Theorem \ref{ritz_estim} let in addition $\ff$ be an $A$-invariant subspace
of $\hh$ corresponding to the top (or bottom) part of the spectrum of $A$. Then
\[
\dist(\Si(A|_{\ff}), \,  \Si((P_{\gg}A)|_{\gg})) \leq
(\max\{\Sigma(A)\} - \min\{\Sigma(A)\})\,\gap^2(\ff,\gg).%\|P_{\ff}-P_{\gg}\|^2.
\]
\end{theorem}
\begin{proof}
As the subspace $\ff$ is $A$-invariant and $A$ is selfadjoint, the subspace $\fo$ is 
also $A$-invariant, %and both  $\ff$ and  $\fo$ are reducing subspaces, 
so $A=P_{\ff}AP_{\ff}+P_{\fo}AP_{\fo}$ and, with a slight abuse of the notation, 
$A=A|_{\ff}+A|_{\fo},$ 
corresponding to the decomposition $\hh=\ff \oplus \fo$, thus 
$\Sigma(A)=\Sigma(A|_{\ff})\cup\Sigma(A|_{\fo})$.
We assume that $\ff$ corresponds to the top part of 
the spectrum of $A$---the bottom part case can be treated by replacing $A$ with $-A$. 
Under this assumption, we have 
$\max\{\Sigma(A|_{\fo})\}\le\min\{\Sigma(A|_{\ff})\}$.

Let us also notice that the inequality we want to prove is unaltered by replacing 
$A$ with $A-\alpha I$ where $\alpha$ is an arbitrary real constant. 
Later in the proof we need $A|_{\ff}$ to be nonnegative.   
We set $\alpha=\min \{\Sigma(A|_{\ff})\}$ and substitute  
$A$ with $A-\alpha I$, so now 
$\max\{\Sigma(A|_{\fo})\}\le 0=\min\{\Sigma(A|_{\ff})\}$, 
thus
\[
\left\|A|_{\ff}\right\|=
\max\{\Sigma(A|_{\ff})\}=
\max\{\Sigma(A)\},
\text{ and }
\left\|A|_{\fo}\right\|=
-\min\{\Sigma(A|_{\fo})\}=-\min\{\Sigma(A)\}.
\]
The constant in the bound we are proving then takes the following form:
\begin{equation}\label{my_constant}
\max\{\Sigma(A)\}-\min\{\Sigma(A)\}=\left\|A|_{\ff}\right\|+\left\|A|_{\fo}\right\|.
\end{equation}

As in the proof of Theorem \ref{ritz_estim},
if $\gap(\ff,\gg)=1$ then the assertion holds since the both
spectra are subsets of $[\min\{\Sigma(A)\},\max\{\Sigma(A)\}]$.
Consequently we can assume
without loss of generality that $\gap(\ff,\gg)<1$. Then we
have $\gg=W \ff$ with $W$ defined by \eqref{dav_ka_ka}. Operators
$\left(W^*(P_{\gg} P_\ff A P_\ff)|_{\gg}W\right)|_{\ff}$ and
$(P_{\gg} P_\ff A P_\ff)|_{\gg}$
are unitarily equivalent, since $W$ is an isometry on $\ff$, thus their spectra are the same.
Now, instead of \eqref{ineq_77}, we use the triangle inequality for the Hausdorff distance:
\begin{multline}\label{triangH}
\dist(\Si(A|_{\ff}),\,\Si((P_{\gg}A)|_{\gg}))\\
\leq
\dist(\Si((A|_{\ff}),\,\Si(\left(W^*(P_{\gg} P_\ff A P_\ff)|_{\gg}W\right)|_{\ff}))\\
+\dist(\Si((P_{\gg} P_\ff A P_\ff)|_{\gg})),\,\Si((P_{\gg}A)|_{\gg})).
\end{multline}

The operator $\sqrt{P_{\ff} P_{\gg} P_{\ff}}|_{\ff}=\sqrt{(P_{\ff} P_{\gg})|_{\ff}}$
is selfadjoint and its smallest point of the spectrum is 
%$\min\left\{\cos\left(\hat{\Theta}(\ff,\gg)\right)\right\}$
$\min\{\cos(\hat{\Theta}(\ff,\gg))\}$ by Definition \ref{angles_from_cos}, which is positive 
by Theorem \ref{gap_angles} with $\gap(\ff,\gg)<1.$ 
The operator $\sqrt{P_{\ff} P_{\gg} P_{\ff}}|_{\ff}$ is invertible, so  
from the polar decomposition $P_{\gg} P_{\ff} = W \sqrt{P_{\ff} P_{\gg} P_{\ff}}$, which gives 
$P_{\ff} P_{\gg} P_{\ff} = P_{\ff} P_{\gg} W \sqrt{P_{\ff} P_{\gg} P_{\ff}}$, 
we obtain by applying the inverse on the right that 
$(P_{\ff}P_{\gg}W)|_{\ff} =\sqrt{P_{\ff} P_{\gg} P_{\ff}}|_{\ff}=(W^*P_{\gg} P_\ff)|_{\ff}.$
Thus,
\begin{eqnarray*}
\left(W^*(P_{\gg} P_\ff A P_\ff)|_{\gg}W\right)|_{\ff} &=&
\left(\sqrt{P_\ff P_{\gg} P_\ff} A \sqrt{P_\ff P_{\gg} P_\ff}\right)|_{\ff}\\
&=& \sqrt{P_\ff P_{\gg} P_\ff}|_{\ff} \sqrt{A|_{\ff}}
\sqrt{A|_{\ff}} \sqrt{P_\ff P_{\gg} P_\ff}|_{\ff}
\end{eqnarray*}
where the operator $A|_{\ff}$ is already made nonnegative  
by applying the shift and the substitution.

The spectrum of the product of two bounded operators, one of which is bijective,
does not depend on the order of the multipliers, since both products are similar 
to each other.  
%see, e.g.\ \cite[Excercise 5.26, p. 106]{weidmann}. 
One of our operators, $\sqrt{P_\ff P_{\gg} P_\ff}|_{\ff}$, in the product is bijective, so 
\[
\Si(\left(W^*(P_{\gg} P_\ff A P_\ff)|_{\gg}W\right)|_{\ff}) =
\Si\left(\sqrt{A|_\ff} (P_\ff P_{\gg})|_\ff \sqrt{A|_\ff}|_{\ff}\right).
\]
Then the first term in the triangle inequality \eqref{triangH} for the Hausdorff distance is estimated
using \cite[Theorem 4.10, p. 291]{kato}:
\begin{eqnarray*}
\dist\big(\Si\left(A|_{\ff}\right),\,&{}&
\Si\left(\left(W^*\left(P_{\gg} P_\ff A P_\ff\right)|_{\gg}W\right)|_{\ff}\right)
\big)\\
&=&
\dist\left(\Si\left(A|_{\ff}\right), \,
\Si\left(\sqrt{A|_\ff} (P_\ff P_{\gg})|_\ff \sqrt{A|_\ff}\right)\right)\\
&\leq&
\left\|A|_{\ff}-\sqrt{A|_\ff}(P_\ff P_{\gg})|_\ff\sqrt{A|_\ff}\right\|\\
&=&
\left\|\sqrt{A|_\ff} (P_\ff - P_\ff P_{\gg})|_\ff \sqrt{A|_\ff}\right\|\\
&\leq&
\left\|A|_\ff\right\|\left\|(P_\ff P_{\go})|_\ff\right\|
=\left\|A|_\ff\right\|\left\|P{_\ff}P_{\go}\right\|^2.
\end{eqnarray*}
To estimate the second term in \eqref{triangH},  
we apply again \cite[Theorem 4.10, p. 291]{kato}:
\begin{multline*}
\dist\left(\Si\left.\left((P_{\gg}P_{\ff}AP_{\ff}\right)\right|_{\gg})),\,
\Si((P_{\gg}A)|_{\gg})\right)
\leq
\|(P_{\gg}P_{\ff}AP_{\ff})|_{\gg}-(P_{\gg}A)|_{\gg}\| \\
=\|(P_{\gg}P_{\fo}AP_{\fo})|_{\gg}\| =\|P_{\gg}P_{\fo}A|_{\fo}P_{\fo}P_{\gg}\|
\leq
\|A|_{\fo}\|\|P_{\gg}P_{\fo}\|^2,
\end{multline*}
where $A=P_\ff A P_\ff +P_{\fo}A P_{\fo}$.
Plugging in bounds for both terms in \eqref{triangH} gives 
\[
\dist\left(\Si(A|_{\ff}), \, \Si((P_{\gg}A)|_{\gg})\right)
\leq
\left\|A|_\ff\right\|\left\|P_{\ff}P_{\go}\right\|^2 +
\left\|A|_{\fo}\right\|\left\|P_{\gg}P_{\fo}\right\|^2.
\]
Assumption $\gap(\ff,\gg)<1$ implies that
$\left\|P_{\ff}P_{\go}\right\|=\|P_{\gg}P_{\fo}\|=\gap(\ff,\gg),$
e.g.,\ see \cite[\S I.8, Theorem 6.34]{kato} %\cite[\$ 34, p. 70]{akhiGlaz} 
and cf. Corollary \ref{cor:gap}.
Thus we obtain
\[
\dist\left(\Si((P_{\ff}A)|_{\ff}), \, \Si((P_{\gg}A)|_{\gg})\right)
\leq
\left(\left\|A|_{\ff}\right\|+\|A|_{\fo}\|\right)\gap^2(\ff,\gg). 
\]
Taking into account \eqref{my_constant} completes the proof.
\end{proof}

We conjecture that our assumption on the invariant subspace representing a specific part 
of the spectrum of $A$ is irrelevant, i.e.,\ the statement of Theorem \ref{ritz_estim_sq} 
holds without it as well, cf. \citet{akpp06,ka07}.
%%%%%%%%%%%%%%%%%%%%%%%%%%%%%%%%%%%%%%%%%%%%%%%%%%%%%%%%%%%%%%%%%%%%%%%%%%%%%%%%%%%%%
\section{The ultimate acceleration of the alternating projectors method}\label{DDM}
Every selfadjoint nonnegative non-expansion $A, \, 0\leq A \leq I$ in
a Hilbert space $\hh$ can be extended to an orthogonal projector in the space
$\hh\times \hh$, e.g., \cite{halmos69,RiSN}, and, thus, can be implicitly written as
(strictly speaking is unitarily equivalent to)
a product of two orthogonal projectors $\fg$ restricted to a subspace
$\ff \subset \hh \times \hh$.
Any iterative method that involves as a main step a multiplication of
a vector by $A$ can thus be called ``an alternating projectors'' method.

In the classical alternating projectors method,
it is assumed that the projectors are given explicitly and that the iterating procedure is trivially
\begin{equation} \label{e:apm}
e^{(i+1)} = \fg e^{(i)}, \,  e^{(0)} \in \ff.
\end{equation}
If $\left\|\left.(\fg)\right|_\ff \right\| < 1$ then the sequence of vectors
 $ e^{(i)}$ evidently converges to zero. Such a situation is typical when
$ e^{(i)}$ represents an error of an iterative method,
e.g., in a multiplicative DDM, and formula
(\ref{e:apm}) describes the error propagation
as in our DDM example below.

If the subspace $\moo = \ff \cap \gg$ is nontrivial and
$\left\|\left.(\fg)\right|_{\ff\ominus\moo}\right\|<1$
then the sequence of vectors $ e^{(i)}$ converges to
the orthogonal projection $e$ of $ e^{(0)}$ onto $\moo$.
The latter is called a von Neumann-Halperin (\cite{MR0032011,MR0141978}) method
in \cite{MR1990157}
of alternating projectors for determining the best approximation to $ e^{(0)}$
in $\moo$. We note that, despite the non-symmetric appearance of the
error propagation operator $ \fg$ in (\ref{e:apm}), it can be
equivalently replaced with the selfadjoint operator $\left.(\fg)\right|_\ff$
since $ e^{(0)} \in \ff$ and thus all $ e^{(i)} \in \ff$.

Several attempts to estimate and accelerate the convergence of
iterations (\ref{e:apm}) are made, e.g.,\
\cite{MR1853223,MR1990157,MR1896233}.
Here, we use a different approach, cf., e.g.,\  \cite{MR992008,MR1009556}, to
suggest the ultimate acceleration of the alternating projectors method.
First, we notice that the limit vector $e \in \moo$ is a nontrivial solution of
the following homogeneous equation
\begin{equation} \label{e:ap1}
\left.(I-\fg)\right|_\ff e=0,\quad  e\in\ff.
\end{equation}
Second, we observe that the linear operator is
 selfadjoint and nonnegative in
the equation above, therefore, a conjugate gradient (CG) method
can be used to calculate approximations to the solution $e$ in the null-space.
The standard CG algorithm for linear systems $Ax=b$
can be formulated as follows, see, e.g., \cite{MR561510}:\\
Initialization: set $\gamma = 1$ and compute the initial residual
$r = b - Ax$; \\*
Loop until convergence:\\*
$\gamma_{old} = \gamma, \, \gamma = (r,r)$; \\*
on the first iteration:  $p = r$; otherwise:\\*
 $\beta = {\gamma}/{\gamma_{old}}$ (standard) or
 $\beta = {(r-r_{old},r)}/{(r_{old},r_{old})}$ \\*
(the latter is recommended if an approximate application of $A$ is used) \\*
     $p = r + \beta p,$
   $r = Ap,$
  $\alpha = \gamma / (r,p),$
   $x = x + \alpha p,$
   $r = r - \alpha r.$\\*
End loop

It can be applied directly to the homogeneous equation $Ae=0$ with
$A=A^\ast \geq 0$ by setting $b=0$.
We need $A = \left.(I-\fg)\right|_\ff$ for equation (\ref{e:ap1}).
Finally, we note that CG acceleration can evidently
be applied to the symmetrized alternating projectors method
with more than two projectors.

The traditional theory of the CG method for non-homogeneous
equations extends trivially to the computation of the null-space of
a selfadjoint nonnegative operator $A$ and gives the following
convergence rate estimate:
\begin{equation} \label{e:cgcs}
 (e^{(k)}, A  e^{(k)}) \leq \min_{\deg p_k = k,  \, p_k(0)=1}
\sup_{ \lambda \in \Sigma(A)\setminus\{0\}} | p_k(\lambda)|^2
\quad (e^{(0)}, A  e^{(0)}).
\end{equation}
For equation (\ref{e:ap1}), $A = \left.(I-\fg)\right|_\ff$ and thus
$ (e^{(k)}, A  e^{(k)} ) = \| \pgo e^{(k)} \|^2$ and
by Definition \ref{angles_from_cos} we have
$\Sigma(A) = 1- \cos^2 \hat{\Theta}(\ff, \gg)$.
Estimate (\ref{e:cgcs}) shows convergence
if and only if zero is an isolated point of the
spectrum of $A$, or, in terms of the angles,
if and only if zero is an isolated point,
or not present,
in the set of angles $\hat{\Theta}(\ff, \gg)$, which is
the same as the condition for convergence of the
original alternating projectors method (\ref{e:apm}), stated above.

Method (\ref{e:apm}) can be equivalently reformulated as
a simple Richardson iteration
\[
e^{(k)} = (I-A)^k e^{(0)},\, e^{(0)} \in \ff, \text{ where } A = \left.(I-\fg)\right|_\ff,
\]
and thus falls into the same class of polynomial methods as does the CG method.
It is well known that the CG method provides the smallest value
of the energy (semi-) norm of the error, in our case of
$\| \pgo e^{(k)} \|$, where $e^{(k)} \in \ff$, which gives us an opportunity
to call it the ``ultimate acceleration'' of the alternating projectors method.

A possible alternative to equation (\ref{e:ap1}) is
\begin{equation} \label{e:ap2}
( \pfo + \pgo  ) e  = 0,
\end{equation}
so we can take $A = \pfo + \pgo$ in the CG method for equation (\ref{e:ap2})
and then $\Sigma(A)$ is given by Theorem \ref{sum_of_projs_Jan}.
Equation (\ref{e:ap2}) appears in the so-called additive DDM method, e.g., \cite{MR2104179}.
A discussion of (\ref{e:ap2}) can be found in \cite[\S 7.1, p. 127]{juju_thesis}.

Estimate (\ref{e:cgcs}) guarantees the finite
convergence of the CG method if the spectrum of $A$
consists of a finite number of points. At the same time,
the convergence of the Richardson method can be slow in
such a case, so that the CG acceleration is particularly noticeable.
In the remainder of the section, we present a simple domain decomposition example
for the one dimensional diffusion equation.

Consider the following one dimensional diffusion equation
$\int_0^1 u^\prime v^\prime dx =
\int_0^1 f v^\prime dx,$ $\forall v\in H_0^1([0,1])$
with the solution $u \in H_0^1([0,1])$, where $H_0^1([0,1])$ is
the usual Sobolev space of real-valued functions with the 
Lebesgue integrable squares of the first generalized derivatives 
and with zero values at the end points of the
interval $[0,1]$. We use the bilinear form
$\int_0^1 u^\prime v^\prime dx$ as a scalar product on
$H_0^1([0,1])$.

We consider DDM with an overlap, i.e.,\ we split
$[0,1] = [0, \alpha] \cup [\beta, 1]$, with
$0 < \beta < \alpha < 1$ so that $[\beta, \alpha]$ is an overlap.
We directly define orthogonal complements:
\[\fo =\{ u \in H_0^1([0,1]): u(x)=0, x \in [\alpha, 1] \}
\text{ and }
\go =\{ v \in H_0^1([0,1]): v(x)=0, x \in [0, \beta] \}
\]
of subspaces $\ff\subset H_0^1([0,1])$ and 
$\gg\subset H_0^1([0,1])$. 
Evidently, $\hh = \fo + \go$, where the sum is not direct
due to the overlap.

It can be checked easily that the subspace $\ff$ consists of
functions, which are linear on the interval $[0,\alpha]$
and the subspace $\gg$ consists of
functions, which are linear on the interval $[\beta,1]$.
Because of the overlap $[\beta,\alpha]$, the intersection
$\moo = \ff \cap \gg$ is trivial and the only solution of
(\ref{e:ap1}) and (\ref{e:ap2}) is $e=0$.

We now completely characterize all angles between
$\ff$ and $\gg$. Let $f \in \ff$ be linear on intervals $[0,\alpha]$ and
$[\alpha, 1]$. Similarly, let $g \in \gg$ be 
linear on intervals $[0, \beta]$ and $[\beta,1]$.
It is easy to see, cf. \cite[\S 7.2]{juju_thesis}, 
that all functions in the subspace $\ff \ominus \Span \{ f\}$
vanish outside of the interval $[\alpha, 1]$, while
all functions in the subspace $\gg \ominus \Span \{ g\}$
vanish outside of the interval $[0, \beta]$. Therefore,
the subspaces $\ff \ominus \Span \{ f\}$
and $\gg \ominus \Span\{ g\}$ are orthogonal, since
$\beta < \alpha$.
We summarize these results
in terms of the principal subspaces:
$\hat\Theta(\ff,\gg)=\Theta(\ff,\gg)=\Theta_p(\ff,\gg) =
\theta(f,g)\cup{\pi}/{2},$ 
where 
$\cos^2\theta(f,g)=(\beta(1-\alpha))/(\alpha(1-\beta)),$
(the latter equality can be derived by elementary calculations, 
see \cite[Theorem 7.2, p. 131]{juju_thesis});
$\Span \{ f\}$ and $\Span \{ g\}$ is one pair of
principal subspaces and
$\ff \ominus \Span \{ f\}$  and $\gg \ominus \Span \{ g\}$
is the other, corresponding to the angle $\pi/2.$

In multiplicative Schwarz
DDM with an overlap
%suggested by Schwarz \cite{Schwarz:1890:GMA},
for two subdomains,
the error propagation of a simple iteration
is given by (\ref{e:apm}) and the convergence rate
is determined by the quantity
\[
\left\|\left.(\fg)\right|_\ff \right\|
= \cos^2 \theta(f,g)=(\beta(1-\alpha))/(\alpha(1-\beta))
< 1,
\] which approaches one when the overlap
$\alpha-\beta$ becomes small. At the same time, however,
the CG method described, e.g., in \cite{MR1009556,MR992008},
converges at most in two iterations, since
the spectrum of $A$ in (\ref{e:ap1})
consists of only two eigenvalues,
$ 1 - \cos^2 \theta(f,g)= \sin^2 \theta(f,g)$ and $1.$

In the additive DDM the error is determined by (\ref{e:ap2})
and the spectrum of $A$, the sum of two orthoprojectors,
by analogy with Theorem \ref{sum_of_projs_Jan}  consists of four eigenvalues,
\[ 1 - \cos \theta(f,g)= 2 \sin^2 (\theta(f,g)/2),\,
1,\,  1 + \cos \theta(f,g)= 2 \cos^2 (\theta(f,g)/2),
\text{ and } 2, 
\]
therefore the  CG method converges at most in four iterations.
Similar results for a finite difference discretization
of the 1D diffusion equation can be found in \cite{MR2058877}. 
%For the 2D diffusion equation,
%we can replace our 1D piecewise linear functions with
%2D piecewise harmonic functions, but the space of traces on the
%interfaces separating the subdomains becomes infinite dimensional,
%so the finite convergence of the CG method is lost.
%%%%%%%%%%%%%%%%%%%%%%%%%%%%%%%%%%%%%%%%%%%%%%%
\section*{Acknowledgements} 
We thank Ilya Lashuk for contributing to the proofs of 
Theorems \ref{andrew_ilya} and \ref{ritz_estim}.
%\bibliographystyle{amsplain}
\bibliographystyle{plainnat}
%\bibliography{abram}

\def\cprime{$'$}

\end{document}